\documentclass[nobysame]{amsart}

\usepackage{amsmath}
\usepackage{amsfonts}
\usepackage{amssymb}
\usepackage{amsthm}
\usepackage{array}
\usepackage{bbm}
\usepackage{chngpage}
\usepackage{comment}
\usepackage{float}
\usepackage[OT2,T1]{fontenc}
\usepackage{graphicx,caption}
\usepackage[export]{adjustbox}
\usepackage{longtable}
\usepackage{mathrsfs}
\usepackage{mathtools}
\usepackage{wrapfig}
\usepackage{cutwin}
\usepackage{shapepar}
\usepackage{tikz}
\usepackage[lowtilde]{url}
\usepackage{subcaption}
\usetikzlibrary{matrix,arrows}
\usepackage{tabularx}
\usepackage{multirow}
\definecolor{citeclr}{rgb}{0.55, 0.55, 0.64}
\definecolor{linkclr}{rgb}{0, 0.21, 0.9447}
\usepackage[colorlinks,
  linkcolor=linkclr,
  citecolor=citeclr,
  urlcolor=blue!46!cyan,
  ]{hyperref}
\usepackage[nameinlink, nosort]{cleveref}
\usepackage{enumitem}
\usepackage{diagbox}
\usepackage{etoolbox}
\usepackage{algpseudocode}
\usepackage{etoolbox}
\usepackage{suffix}
\usepackage{nicefrac}

\usepackage[margin=0.9in]{geometry}

\patchcmd{\subsection}{-.5em}{.5em}{}{}
\patchcmd{\subsubsection}{-.5em}{.5em}{}{}

\patchcmd{\section}{\normalfont}{\normalfont\Large}{}{}

\newcommand{\CC}{\mathbb{C}}
\renewcommand{\C}{\mathbb{C}}

\newcommand{\R}{\mathbb{R}}
\newcommand{\Z}{\mathbb{Z}}

\newcommand{\cF}{\mathcal{F}}

\newcommand{\cM}{\mathcal{M}}

\newcommand{\cR}{\mathcal{R}}
\newcommand{\cS}{\mathcal{S}}

\DeclareSymbolFont{cyrletters}{OT2}{wncyr}{m}{n}
\DeclareMathSymbol{\sha}{\mathalpha}{cyrletters}{"58}

\newcommand{\eps}{\varepsilon}
\newcommand{\dee}{\partial}

\renewcommand{\Re}{\mathrm{Re}}
\renewcommand{\Im}{\mathrm{Im}}
\renewcommand{\arg}{\mathrm{arg}}
\newcommand*{\onesymb}{\text{\large\usefont{U}{bbold}{m}{n}1}} 
\newcommand{\one}[1]{\raisebox{-0.33pt}{\onesymb}\mspace{-1.5mu}\{#1\}}
\newcommand{\onelr}[1]{\raisebox{-0.33pt}{\onesymb}\mspace{-4.5mu}\left\{#1\right\}}
\renewcommand{\mod}{\,\mathrm{mod}\,}

\newcommand{\ds}{\displaystyle}
\newlength{\strutheight}
\newcommand{\half}{\frac{1}{2}}
\newcommand{\thalf}{\tfrac{1}{2}}

\newcommand\starsum{\mathop{\sum\nolimits^{*}}}
\WithSuffix\newcommand\starsum*{\mathop{\sum\nolimits^{\mathrlap{*}}}}
\newcommand{\halfGamma}[1]{\Gamma\!\left(\frac{#1}{2}\right)\!}

\newcommand{\x}{x}
\newcommand{\rep}{\varphi}
\newcommand{\repdim}{r}

\newcommand{\rootnum}{\omega}

\newcommand{\qrep}{q_\rep}

\newcommand{\wtpsi}{\mspace{7.5mu}\lower0.18ex\hbox{$\widetilde{}$}\mkern-7.5mu \psi}

\newtheorem{theorem}{Theorem}[section]
\newtheorem{lemma}[theorem]{Lemma}
\newtheorem{corollary}[theorem]{Corollary}
\newtheorem{conjecture}[theorem]{Conjecture}
\newtheorem{proposition}[theorem]{Proposition}

\theoremstyle{definition}
\newtheorem{definition}[theorem]{Definition}
\newtheorem{remark}[theorem]{Remark}
\newtheorem{figurecap}[theorem]{Figure}
\newtheorem{tablecap}[theorem]{Table}

\author{Alex Cowan}
\address{Department of Mathematics, University of Waterloo, Waterloo, ON, Canada}
\email{alex.cowan@uwaterloo.ca}
\thanks{The author was partially supported by the Simons Foundation Collaboration Grant 550031.}

\title{Murmurations and ratios conjectures}
\date{\today}

\AtBeginDocument{%
   \def\MR#1{}
}

\begin{document}
\begin{abstract}

  We introduce a new method for studying murmurations, based on random matrix theory.
  With this method, we exhibit murmurations or similar phenomena:
  assuming ratios conjectures, for elliptic curves ordered by height, quadratic twists of a fixed elliptic curve, and the inverse Mellin transform of the shifted second moment of $\zeta'/\zeta$ on vertical lines;
  assuming GRH, for primitive quadratic Dirichlet characters, and holomorphic modular forms of prime level tending to infinity with sign and weight fixed;
  and unconditionally the inverse Mellin transform of the shifted second moment of $\zeta$ on vertical lines.
  We also present a generalization of our approach which relies only on the approximate functional equation in place of ratios conjectures.
\end{abstract}
\maketitle
{
\tableofcontents
}
\section{Introduction}


Unexpected and striking oscillations in the average $a_p$ values of a set of elliptic curves as $p$ varies were observed empirically in \cite{HLOP} using techniques from data science. Since then, similar patterns have been discovered for many other types of arithmetic objects \cite{drew_letter}. These sorts of phenomena are called \textit{murmurations}, and characterized by invariance under simultaneous scaling of $p$ and the conductors of the arithmetic objects. Proofs involving trace formulas have appeared for Dirichlet characters \cite{LOP}, holomorphic weight $k$ cusp forms \cite{zubrilina}, holomorphic level $1$ cusp forms \cite{bblld}, and level $1$ Maass forms \cite{blldshz}. However, the situation for elliptic curves, despite being the catalyst for this other work, remains mysterious: as of yet there are no theorems, conjectures, or even heuristic predictions for the experimental results of \cite{HLOP}. 

In this paper, we prove the following, which describes murmurations of elliptic curves 
in the height-ordered family
\begin{align*}
  \cF(H) \coloneqq \left\{E_{a,b}\mspace{-2mu}:\mspace{-2mu} y^2 = x^3 + ax + b \;:\; |a| \leqslant H^{\frac{1}{3}},\,\, |b| \leqslant H^{\half},\,\, 3\nmid a,\,\, 2\nmid b,\,\, p^4 \mid a \Rightarrow p^6 \nmid b\right\}
  .
\end{align*}
\begin{samepage}
\begin{theorem}\label{elliptic_curve_murmuration_rho}
  Fix $\rootnum \in \{\pm 1\}$ and define $\cF(H)^\rootnum \coloneqq \{E \in \cF(H) \,:\, \rootnum_E = \rootnum\}$. Assume that \eqref{conductor_distribution}, \eqref{small_conductors}, and the ``ratios conjecture'' \cite[Conj.\ 3.7]{DHP} hold with $\cF(H)$ replaced with $\cF(H)^\rootnum$.
  For any $H, y, T, \eps$ such that $0 < \eps < \tfrac{5}{6}$ and $(Hy)^{\half + \eps} \ll T < Hy$,
\begin{align*}
  &\frac{1}{\#\cF(H)^\rootnum}\sum_{E \in \cF(H)^\rootnum}\frac{1}{\sqrt{Hy}}\sum_{\substack{p^k < Hy \\ p\, \nmid N_E}} \left(\alpha_p^k + \overline{\alpha}_p^k\right)\log p \\
  &\hspace{3cm}= \frac{\rootnum}{2\pi i} \int_\R\int_{\half + \eps - iT}^{\half + \eps + iT} \frac{\Gamma(\frac{3}{2} - s)}{\Gamma(\half + s)} \zeta(2s) A(\thalf - s, s - \thalf)  \left(4\pi^2 \frac{y}{\lambda}\right)^{s-\half}\frac{ds}{s} \,F_N'(\lambda)\,d\lambda\\
  &\hspace{3.4cm} - \frac{1}{\sqrt{Hy}}\sum_{p^k < \sqrt{Hy}} \log p + O\!\left(H^{\frac{5}{12} + \eps}\#\cF(H)^{-1}y^\eps T^\eps + (\log H)^{-\frac{5}{6}}\right),
\end{align*}
where $A(\alpha,\gamma)$ is the product from \cref{A_def} involving traces of Hecke operators, and $F_N'$ is the approximation given in \cref{FN_def} to the distribution of $N_E/H$ among $E \in \cF(H)$.
\end{theorem}
\noindent
Note that the first term on the right hand side of \cref{elliptic_curve_murmuration_rho} does not depend on $H$. 
\end{samepage}

Our methods can be applied to many sorts of arithmetic objects, not just elliptic curves. In the case of quadratic Dirichlet characters, we prove the following.

\begin{theorem}\label{dirichlet_murmuration}
  Assume the ``ratios conjecture'' \cite[Conj.\ 2.6]{conrey_snaith}. Fix $1 < D_0 < D$, and define
  \begin{align*}
    \cF_\chi \coloneqq \{d \,:\, D_0 < d < D,\, \text{\emph{$d$ a fundamental discriminant}}\}.
  \end{align*}
  For $\eps$ such that $\tfrac{1}{\log \x } < \eps < \tfrac{1}{4}$, and $T$ such that $T \ll \x^{1-\eps}$ and $T \ll D_0^{1-\eps}$,
  \begin{align*}
    \frac{1}{\#\cF_\chi}\sum_{d \in \cF_\chi} \frac{1}{\sqrt{\x}} \sum_{\substack{p^k < \x  \\ \text{$k$ \emph{odd}}}} \chi_d(p)\log p
    ={}&
    \frac{1}{2\pi i}\int_{\half + \eps - iT}^{\half + \eps + iT}
    \frac{\pi^2}{6}\frac{\Gamma\!\left(\frac{1-s}{2}\right)}{\Gamma\!\left(\frac{s}{2}\right)} \frac{\zeta(2-2s)}{\zeta(3-2s)}
    \frac{1}{\#\cF_\chi} \sum_{d \in \cF_\chi} \left(\frac{\pi \x }{d}\right)^{\!s - \half} \,\frac{ds}{s}\\
    &+
    O\!\left(D^{\half} \#\cF_\chi^{-1} + \x^\half T^{-1} 
    \right)(\x DT)^\eps.
  \end{align*}
\end{theorem}

\Cref{fig:kronecker_murmuration} illustrates \cref{dirichlet_murmuration} for $\cF_\chi \coloneqq \{d \,:\, 95,\!000 < d < 105,\!000,\, \text{$d$ a fundamental discriminant}\}$ by plotting, as functions of $\x $, the left and right hand sides of
\begin{align}
  \label{eq:kronecker_fig}
  &\frac{1}{\#\cF_\chi}\sum_{d \in \cF_\chi}\frac{1}{\x^\half}\sum_{\substack{p^k < \x  \\ \text{$k$ odd}}} \chi_d(p)\log p 
  \,\approx\,
  \frac{1}{2\pi i}\int_{\half + \eps - iT}^{\half + \eps + iT} \frac{\pi^2}{6}\frac{\Gamma\!\left(\frac{1-s}{2}\right)}{\Gamma\!\left(\frac{s}{2}\right)} \frac{\zeta(2-2s)}{\zeta(3-2s)} \frac{1}{\#\cF_\chi} \sum_{d \in \cF_\chi} \left(\frac{\pi \x }{d}\right)^{\!s - \half} \frac{ds}{s}.
\end{align}
\begin{figure}[H]
  \includegraphics[width=0.95\textwidth]{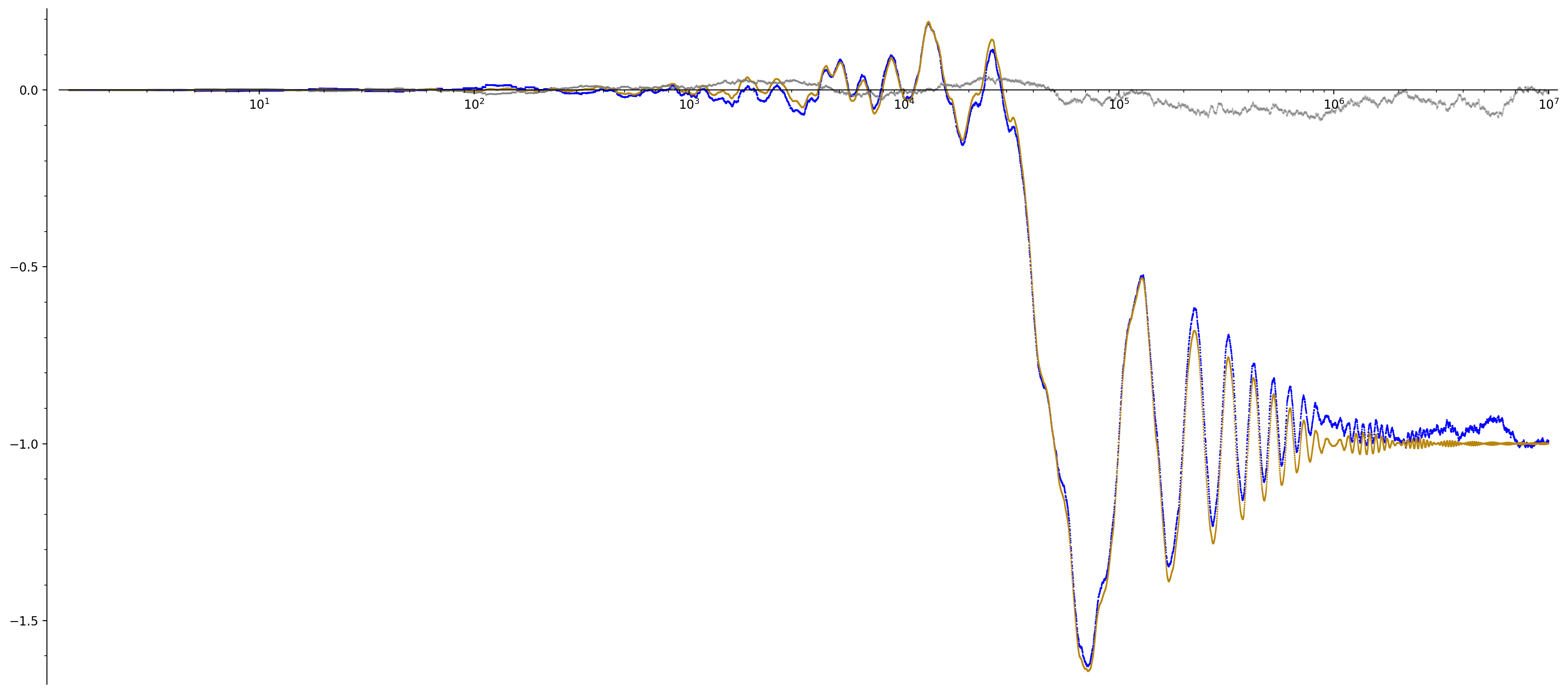}
\begin{figurecap}\label{fig:kronecker_murmuration}
  For $T = 900$ and $\eps = 0.1$, the left and right hand sides of \eqref{eq:kronecker_fig} in blue and gold respectively, as well as their difference in grey, as functions of $\x $.
  The integral in \eqref{eq:kronecker_fig} is approximated by Riemann sum evaluated at $180,\!000$ equally-spaced points. In this example $\#\cF_\chi = 3038$. Our code is available at \cite{github}.
\end{figurecap}
\vspace{-\baselineskip}
\end{figure}

Both \cref{elliptic_curve_murmuration_rho,dirichlet_murmuration} give asymptotics
\begin{align*}
  \frac{1}{\#\cF(N)}\sum_{\rep \in \cF(N)} \frac{1}{\sqrt{\#\cS(\x)}} \sum_{q \in \cS(\x)} a_\rep(q) \,\sim\, \psi\!\left(\frac{\x}{N}\right)
\end{align*}
in which
the left hand side is an appropriately normalized sum of Dirichlet coefficients $a_\rep(q)$ over sets $\cF(N)$ and sets $\cS(\x)$ of arithmetic objects $\rep$ with conductor at most $N$ and of prime powers $q$ at most $\x$ respectively,
and the right hand side is a function of only the ratio $\x/N$ --- i.e.\ is scale-invariant --- as $\x, N \to \infty$.
An essential characteristic is that the average is taken with the horizontal and vertical aspects, governed by $\x$ and $N$, coupled.
We'll call such asymptotics \textit{murmurations}.
Variations in which weights are included, in either aspect, will also be called murmurations.

A different template for murmurations is used in \cite{zubrilina,LOP}, based on a definition of Sarnak's \cite{sarnak_letter} for a distribution called a ``murmuration density''. In \cref{thm:sarnak} we give the formula for converting between our template and theirs.

\Cref{elliptic_curve_murmuration_rho,dirichlet_murmuration} hinge on estimates of distributions of low-lying zeros in families of $L$-functions. There are many existing results on this subject, which can in turn be used to produce murmurations. For example, recent work of \v Cech \cite{cech} yields murmurations for quadratic Dirichlet characters under the assumption of GRH instead of \cite[Conj.\ 2.6]{conrey_snaith}:

\begin{theorem}\label{cech_murmuration}
  Assume GRH. Let $f$ be a smooth, fast-decaying 
  weight function with Mellin transform $\cM f$, and let $\chi_n$ denote the Jacobi symbol $(\frac{\cdot}{n})$. For any $N, y, T, c$ such that $2 < T < Ny$ and $\thalf + \eps < c < \tfrac{3}{4}$,
  \begin{align*}
    \sum_{\substack{n \,>\, 0 \\ n \text{\emph{ odd, squarefree}}}} &\frac{1}{N}f\!\left(\frac{n}{N}\right) \frac{1}{\sqrt{Ny}}\sum_{p^k < Ny} \chi_n(p^k)\log p\\
    ={} &\frac{1}{2\pi i}\int_{c-iT}^{c+iT} \frac{1}{4} \left(\frac{\Gamma\!\left(\frac{1-s}{2}\right)}{\Gamma\!\left(\frac{s}{2}\right)} + \frac{\Gamma\!\left(\frac{2-s}{2}\right)}{\Gamma\!\left(\frac{1+s}{2}\right)}\right)\frac{\zeta(2-2s)}{\zeta^{(2)}(3-2s)} (\pi y)^{s - \half} \cM f(\tfrac{3}{2}-s)\frac{ds}{s}\\
    &+ 
    \frac{2 \cM f(1)}{3\zeta(2)}
    + O\!\left(N^{\half - c} y^{c-\half}  + N^{\half} y^{\half} T^{-1}\right)(NyT)^\eps
    .
  \end{align*}
\end{theorem}

Similarly, work Miller and Montague \cite{miller_montague} can be used to produce murmurations for newspaces of holomorphic modular forms in the level aspect under GRH:

\begin{theorem}
  \label{newspace_murmuration_GRH}
  Assume GRH.
  Let $H_k^\pm(N)$ denote the set of normalized eigenforms of weight $k$, level $N$, and root number $\pm 1$. Let $\lambda_f(n)$ be such that $\lambda_f(p)p^{\frac{k-1}{2}}$ is the $p^{\text{th}}$ Hecke eigenvalue of $f \in H_k^\pm(N)$. Let $g$ be an even Schwartz function whose Fourier transform $\hat{g}$ is supported on $(-\sigma,\sigma)$.
  For fixed $\pm \in \{+,-\}$ and $k \in 2\Z_{>0}$, as $N \to \infty$ prime,
  \begin{align*}
    &\sum_{f \in H_k^\pm(N)} \frac{\Gamma(k-1)}{(4\pi)^{k-1}\langle f,f \rangle} \sum_{p \neq N} \hat{g}\!\left(\frac{\log p}{\log N}\right) \frac{\lambda_f(p) \log p}{\sqrt{p}}\\
    &\hspace{0.5cm}=\,\,\pm 2\lim_{\delta \to 0^+}\int_0^\infty \hat{g}\!\left(1 + \frac{\log y}{\log N}\right) \int_{-\infty}^\infty \frac{\Gamma(\frac{k}{2} - 2\pi it)}{\Gamma(\frac{k}{2} + 2\pi it)} \prod_{p} \left( 1 + \frac{1}{(p-1)p^{4\pi it + \delta}} \right) \left(4\pi^2 y\right)^{2\pi it} dt\, \frac{dy}{y} + O\!\left(N^{\frac{\sigma}{2} - 1 + \eps}\right).
  \end{align*}
\end{theorem}
Taking $\hat{g}(\xi)$ to be sharply peaked around $\xi = 1$ and $y$ to be fixed, \cref{newspace_murmuration_GRH} yields murmurations.

\Cref{elliptic_curve_murmuration_GRH}, based on work of Huynh--Miller--Morrison \cite{hmm}, is another application of our method, giving murmurations for quadratic twists of elliptic curves assuming \cite[Conj.\ 2.1]{HKS}.

The estimates for the distributions of low-lying zeros in families of $L$-functions our method takes as input 
in most cases come from random matrix theory, and specifically \textit{ratios conjectures}. These conjectures predict zero distributions very precisely, including the lower order terms needed to observe murmurations, and are highly nontrivial. The connection to murmurations comes via \textit{explicit formulas}: identities in analytic number theory which relate $L$-function zeros to partial sums of their prime-power Dirichlet coefficients. This approach to studying murmurations has so far not appeared in the literature outside of our short note \cite{cowan} and our paper \cite{quadratictwists} accompanying the present one; existing murmurations results have instead relied on trace formulas.

There is a ``recipe'' \cite{conrey_snaith} for producing ratios conjectures in a wide variety of settings, and they are of significant interest in the field of random matrix theory. Because of this, there is extensive evidence to support them --- see e.g.\ \cite{rubinstein:2013}, as well as the confirmations under GRH \cite{miller_montague} and \cite{cech} mentioned above --- and there are many settings in which the recipe has been implemented to obtain predictions of zero distributions: in addition to the results we leverage in this paper, see e.g.\ \cite{miller09,gjmmnpp,miller_peckner,fiorilli_miller}, to which the method presented here may readily be applied.

A closely connected subject in random matrix theory to zero distributions is that of shifted moments of $L$-functions or their logarithmic derivatives along vertical lines. The ratios conjecture recipe is applicable in such settings. Leveraging structural similarities with the examples above involving sums of Dirichlet coefficients, we exhibit murmuration-like phenomena for correlations of $\zeta'/\zeta$ along vertical lines.


\begin{theorem}
  \label{thm:zeta_logderiv_correlation_murmurations}
  Assume the ratios conjecture \cite[Conj.\ 2.1]{conrey_snaith}.
  For $T, \Delta T, H, c > 0$ and $\x  > 1 + \eps$,
  \begin{align*}
    \frac{1}{2\pi i}
    \int_{c - iH}^{c + iH}
    &\int_T^{T+\Delta T} \frac{\zeta'}{\zeta}\!\left(\half + it + \alpha\right) \frac{\zeta'}{\zeta}\!\left(\half - it - \alpha + z\right) dt \,\x^z \,\frac{dz}{z+1}
    \\
    ={}&
    \frac{1}{2\pi i} \int_{c - iH}^{c + iH}\zeta(1 + z) \zeta(1 - z)
    \prod_p \left(1 - \frac{1}{p^{1 + z}}\right)\!\left(1 - \frac{2}{p} + \frac{1}{p^{1 + z}}\right)\!\left(1 - \frac{1}{p}\right)^{\!-2}
    \int_T^{T+\Delta T} \!\left(\frac{2\pi \x }{t}\right)^{\!z} dt \,\frac{dz}{z+1}\\
    &
    + O\!\left(\x^c H^\eps (T^{\half + \eps} + \Delta T)\right).
  \end{align*}
\end{theorem}
The first term on the right hand side of \cref{thm:zeta_logderiv_correlation_murmurations} exhibits scale invariance in $\x /T$ characteristic of murmurations. It is not at all clear a priori why the inverse Mellin transform of the shifted second moment of $\zeta'/\zeta$ ought to be scale invariant like this.

Random matrix theory has predictions for correlations $\zeta$ along vertical lines as well. Bettin \cite{bettin} has verified these predictions to enough precision for the same sort of murmuration-like phenomenon to be exhibited unconditionally.

\begin{theorem}
  \label{thm:zeta_correlation_murmurations_intro}
  For $T, \Delta T, \x , H, c \in \R$ and $\alpha \in \C$ satisfying
  $T \gg \Delta T > 2H > 0$, $T \geqslant 2$, $\x  > 1$, $|\x - m| > \eps$ for all $m \in \Z$, $H + |\Im(\alpha)| < \thalf T$, $\Re(\alpha) \ll \frac{1}{\log T}$, and $e^{-T} \ll c \ll \frac{1}{\log T}$,
  \begin{align*}
    \frac{1}{2\pi i} \int_{c - iH}^{c + iH}
    &\int_{T - \Im(\alpha)}^{T+\Delta T - \Im(\alpha)}
    \zeta\!\left(\thalf + it + \alpha\right)\zeta\!\left(\thalf - it - \alpha + z\right)\,dt\,\x^z\frac{dz}{z}
    \\
    ={}&
    \frac{1}{2\pi i} \int_{c-iH}^{c+iH} \zeta(1-z) \int_{T + H}^{T + \Delta T - H} \left(\frac{2\pi \x }{t}\right)^z dt\,\frac{dz}{z}
    \,+\, \Delta T \sum_{d < \x } \frac{1}{d}
    \,-\, \one{c < 0}(\log \x  + \gamma)\Delta T
    \\
    &+\, O\!\left(T^\half + H + H^{-1}\Delta T \right)(\x TH)^\eps
    ,
  \end{align*}
  where $\gamma = 0.577\!\:\!...$ denotes the Euler--Mascheroni constant.
\end{theorem}
\Cref{thm:zeta_correlation_murmurations} is a more precise version of \cref{thm:zeta_correlation_murmurations_intro}.

\Cref{sec:kronecker_ratiosconjecture} presents many of the method's main ideas in the simplest case: quadratic Dirichlet characters. The section revolves around proving \cref{dirichlet_murmuration}.
Several comments, remarks, and manipulations the section presents apply throughout the paper. Viz.,
\begin{itemize}
\item
  The presence of an error term in murmurations results is new; the existing results \cite{LOP,zubrilina,bblld,blldshz} only give asymptotics. In the case of \cref{dirichlet_murmuration}, the error term is found to come entirely from the ratios conjecture \cite[Conj.\ 2.6]{conrey_snaith}; see \cref{rem:dirichlet_error}.
\item
  \Cref{rem:dirichet_functional_equation_factor} highlights which piece of the ratios conjecture prediction is the one which ultimately yields murmurations. An analogous piece is present in every example in this paper. The scale invariance in \cref{thm:zeta_logderiv_correlation_murmurations,thm:zeta_correlation_murmurations_intro} can be foreseen from this remark. \Cref{sec:afe} explains why this piece reoccurs in every example, and leads into our companion paper \cite{quadratictwists}.
\item
  \Cref{difference_remark} discusses taking a difference of \cref{dirichlet_murmuration} between two nearby $\x$ values, so as to isolate a short sum over primes.
\item
  \Cref{prop:psi_as_residues} expresses the right hand side of \cref{dirichlet_murmuration} as a sum over zeros of $\zeta$.
\end{itemize}

\Cref{sec:elliptic_curves} is separated in two. \Cref{sec:ec_raw} mirrors the methodology of \cref{sec:kronecker_ratiosconjecture}, and leads up to \cref{elliptic_curve_murmuration}, the analogue of \cref{dirichlet_murmuration}. \Cref{sec:ec_dist} proves \cref{elliptic_curve_murmuration_rho} by combining \cref{elliptic_curve_murmuration} with the distribution of conductors in $\cF(H)$, which was determined for this purpose in \cite{conductor}.
\Cref{elliptic_curve_murmuration_rho} exhibits murmurations for a family in which the conductors are not approximately constant. \Cref{sec:ec_dist} provides a template for murmurations in such families.

\Cref{sec:other_examples} gives several more examples of our method:
\begin{itemize}
\item
  \Cref{sec:kronecker_2} introduces a different template, using Miller's work \cite{miller08} on quadratic Dirichlet characters.
\item
  \Cref{sec:cech} proves \cref{cech_murmuration}.
\item
  \Cref{sec:GRH} proves \cref{newspace_murmuration_GRH}.
\item
  \Cref{sec:elliptic_curve_quadratic_twist} uses work of Huynh--Miller--Morrison \cite{hmm} to show murmurations in families of quadratic twists of elliptic curves assuming ratios conjectures.
\item
  \Cref{sec:zeta_logderiv_correlation} proves \cref{thm:zeta_logderiv_correlation_murmurations}. \Cref{prop:zeta_logderiv_correlation_eval} gives a variation where the left hand side is interpreted as a sum over pairs of zeros.
\item
  \Cref{sec:zeta_correlation} proves the more precise variant \cref{thm:zeta_correlation_murmurations} of \cref{thm:zeta_correlation_murmurations_intro}, which also includes an interpretation of the inverse Mellin transform on the left hand side.
\end{itemize}

\Cref{sec:afe} demonstrates how to construct murmurations using the approximate functional equation. It is a specific term in the approximate functional equation which yields the reoccurring piece in ratios conjectures leading to scale invariance. As an example, it is sketched how to construct murmurations for quadratic characters summed over all integers, as opposed to only prime powers. This endeavour is completed in \cite{quadratictwists}, obtaining for instance:

\begin{theorem}[{Special case of \cite[Thm.\ 1.4]{quadratictwists}}]
  \label{thm:gamma_murmurations}
  Let $\tau$ be a real number and let $\chi$ be an even primitive Dirichlet character (possibly trivial) with conductor $q$. Pick $a \in \Z$ coprime to $q$, and $1 < D_0 < D$. Define
  \begin{align*}
    \cF \coloneqq \{d \,:\, D_0 < d < D,\, \text{\emph{$d$ a fundamental discriminant}},\,d = a \mod q\}.
  \end{align*}
  For any $\tfrac{3}{4} < \delta < 1$,
  in the range
  $(qD)^{1 - \eps} \ll \x \ll (qD)^{1+\eps}$,
  $D^{\delta - \eps} \ll \#\cF \ll D^{\delta + \eps}$,
  and
  $\tau \ll D^{1 - \delta - \eps}$,
  \begin{align*}
    \frac{G(\bar\chi)}{\sqrt{q}}
    \bigg(&\frac{\pi}{q}\bigg)^{i\tau}
    \bar\chi(a) \bigg(\frac{a}{q}\bigg)
    \,\frac{1}{\#\cF} \sum_{d \in \cF}
    \frac{1}{\sqrt\x}
    \sum_{n=1}^\infty
    e^{\nicefrac{-n^2}{x^2}}
    n^{i\tau}\chi(n)\chi_d(n)
    \\
    ={}
    &
    \frac{1}{4\pi i}\int_{\eps - i\infty}^{\eps + i\infty}
    \frac{\halfGamma{s}\halfGamma{1-s+i\tau}}{\halfGamma{s-i\tau}}
    \frac{L(2-2s+2i\tau, \bar{\chi}^2)}{L^{(2)}(3-2s+2i\tau, \bar{\chi}^2)}
    \prod_{\mspace{1.75mu}p \nmid q} \!\left(1 - \frac{1}{(p+1)(1 - \chi^2(p)p^{3-2s+2i\tau})}\right)
    \left(\frac{\pi \x}{q D}\right)^{\!s - \half} ds
    \\
    &
    + O\big(
    q^{\frac{3}{4}}
    D^{-\frac{1}{8} + \left|\delta - \frac{7}{8}\mspace{-2mu}\right|}
    (1 + \one{\delta < \tfrac{7}{8}} \tau^{\frac{1}{4}}) \big)(\tau qD)^\eps
    ,
  \end{align*}
  where $G(\bar\chi)$ is the Gauss sum and $\big(\frac{a}{q}\big)$ is the Kronecker symbol.
\end{theorem}

\section*{Acknowledgements}
We thank Sandro Bettin, Chantal David, Valeriya Kovaleva, Kimball Martin, Steve Miller, Alexey Pozdnyakov, Mike Rubinstein, Nina Snaith, Drew Sutherland, and Jerry Wang for their help.


\section{Murmuration densities}
\label{sec:density}

This section is about \cref{thm:sarnak}, which indicates how to convert between the results in this paper and murmuration densities, a notion introduced by Sarnak in \cite{sarnak_letter}. First some notation.

Let $\starsum*\;$ denote a sum over some subset of positive integers, e.g.\ primes, prime powers, or all integers. The variable of summation will always be denoted $n$. We write $\starsum*_n$ to mean a sum over the entire subset.

Let $\cF_\infty$ denote a family of arithmetic objects $\rep$ with conductor $q_\rep$ and Dirichlet coefficients $a_\rep(n)$. Murmurations results in this paper consider finite subfamilies $\cF(q)$ in which $q_\rep \approx q$ --- our precise requirement is that
\begin{align}
  \label{eq:mellin_1}
  &\frac{1}{\#\cF(q)} \sum_{\rep \in \cF(q)} \frac{1}{\sqrt{qy}} \starsum*_n \,a_\rep(n) \cM^{-1}g\!\left(\frac{n}{qy}\right)
  \sim
  \frac{1}{2\pi i} \int_{c - i\infty}^{c + i\infty} F(s) y^{s-\half} g(s)ds
\end{align}
as $q \to \infty$ for
\begin{itemize}
\item
  some function $F$,
\item
  some real number $c$,
\item
  all real numbers $y$ in an open interval containing $1$, and
\item
  all nice test functions $g$.
\end{itemize}
In \eqref{eq:mellin_1} and elsewhere $\cM$ denotes the Mellin transform (i.e.\ $\cM h(s)$ is the Mellin transform of the function $h$ evaluated at the point $s$), with the inverse $\cM^{-1}$ taken along a suitable vertical line.
%

If, as $q \to \infty$ and $\Delta y \to 0$ jointly in some way,
\begin{align*}
  \cM\!\left\{
  \frac{{\ds\sum_{\rep \in \cF_\infty} \Phi\!\left(\frac{q_\rep}{q}\right) \starsum*_{y < \frac{n}{q} < y + \Delta y} a_\rep(n)\sqrt{n}}}{{\ds\sum_{\rep \in \cF_\infty} \Phi\!\left(\frac{q_\rep}{q}\right) \starsum*_{y < \frac{n}{q} < y + \Delta y} 1}}
  \right\}
  \,\sim\,
  \cM\{M(y)\}\cM\{\Phi(t)t^\delta\} 
\end{align*}
for some distribution $M$ (i.e.\ linear functional acting on test functions), some $\delta \in \R$ depending only on $\cF_\infty$, and all nice test functions $\Phi$, then $M$ is a \textit{murmuration density} in the sense of \cite{sarnak_letter}.

Note that the asymptotic \eqref{eq:mellin_1} is equivalent to
\begin{align*}
  &\cM\!\left\{\frac{1}{\#\cF(q)} \sum_{\rep \in \cF(q)} \frac{1}{\sqrt{qy}} \starsum*_n \,a_\rep(n) \cM^{-1}g\!\left(\frac{n}{qy}\right)\right\}
  \,\sim\,
  F(\thalf - s)g(\thalf - s)
  ,
\end{align*}
which is in the same spirit.

The following theorem gives the conversion between our results and murmuration densities.
\begin{theorem}
  \label{thm:sarnak}
  With notation as above, assume that \eqref{eq:mellin_1} is satisfied. Then the distribution
  \begin{align*}
    \cM^{-1}\mspace{-5mu}\left\{\frac{s}{s-2} F(\thalf - s)\right\}\!(y)
    =
    \frac{1}{2\pi i}\int_{\half - c - i\infty}^{\half - c + i\infty} \frac{s}{s-2} F(\thalf - s) y^{-s}\,ds
  \end{align*}
  is the murmuration density for the family $\cF_\infty$.
\end{theorem}

\begin{proof}
  Let $V$ be a continuously differentiable function of sufficiently quick decay, and let $b(n)$ be a sequence supported on the positive integers.
  By partial summation \cite[Thm.\ 4.2]{apostol},
  \begin{align}
    \nonumber
    \starsum*_{n} \;b(n)\sqrt{n}\cdot\frac{1}{\sqrt{n}} V\!\left( \frac{n}{x} \right)
    &= -\int_0^\infty \frac{\partial}{\partial t} \left[ \frac{1}{\sqrt{t}} V\!\left( \frac{t}{x} \right) \right] \starsum*_{n < t} \;b(n) \sqrt{n} \, dt
    \\
    \nonumber
    &=
    \int_0^\infty \left[ \frac{1}{2t^{\frac{3}{2}}} V\!\left( \frac{t}{x} \right) - \frac{1}{x\sqrt{t}} V'\!\left( \frac{t}{x} \right) \right] \starsum*_{n < t} \;b(n) \sqrt{n} \, dt
    \\
    \nonumber
    &=
    \int_{0}^\infty \frac{1}{\sqrt{t}} \left[ \frac{1}{2} V\!\left( \frac{t}{x} \right) - \frac{t}{x} V'\!\left( \frac{t}{x} \right) \right] \starsum*_{n < t} \;b(n) \sqrt{n} \, \frac{dt}{t}
    \\
    \nonumber
    &=
    \int_{0}^\infty \frac{1}{\sqrt{xt}} \left[ \frac{1}{2} V(t) - t V'(t) \right] \starsum*_{n < xt} \,b(n) \sqrt{n} \, \frac{dt}{t}
    \\
    \label{eq:partial_summation}
    &=
    \int_{0}^\infty \sqrt{t} \left[ \frac{1}{2} V\!\left(\frac{1}{t}\right) - \frac{1}{t} V'\!\left(\frac{1}{t}\right) \right] \frac{1}{\sqrt{x}} \starsum*_{n < \frac{x}{t}} \,b(n) \sqrt{n} \, \frac{dt}{t}
    .
  \end{align}

  Applying \eqref{eq:partial_summation} to \eqref{eq:mellin_1},
  \begin{align}
    \nonumber
    \frac{1}{\#\cF(q)} \sum_{\rep \in \cF(q)}
    &\frac{1}{\sqrt{qy}} \starsum*_{n} \,a_\rep(n) \cM^{-1}g\!\left( \frac{n}{qy} \right)
    \\
    \label{eq:mellin_2}
    &=
    \int_0^\infty \frac{1}{\sqrt{t}} \left[ \frac{1}{2} \cM^{-1}g\!\left( \frac{1}{t} \right) - \frac{1}{t} (\cM^{-1}g)' \!\left( \frac{1}{t} \right) \right]
    \frac{1}{\#\cF(q)} \sum_{\rep \in \cF(q)} \frac{t}{qy} \starsum*_{n < \frac{qy}{t}} a_\rep(n) \sqrt{n} \, \frac{dt}{t}
    .
  \end{align}

  Let
  \begin{align*}
    \cM^{-1}H_q(t) &\coloneqq \frac{1}{\#\cF(q)} \sum_{\rep \in \cF(q)} \frac{1}{qt} \starsum*_{n < qt} a_\rep(n) \sqrt{n}
  \end{align*}
  as a distribution.

  The factor accompanying $\cM^{-1}H_q\big(\frac{y}{t}\big)$ in the integrand of \eqref{eq:mellin_2} is easily seen to be equal to $(2-s)g(\thalf - s)$:
  \begin{align*}
    \int_0^\infty \frac{1}{2\sqrt{t}} \cM^{-1}g\!\left( \frac{1}{t} \right) t^s \, \frac{dt}{t}
    &= \half\int_0^\infty \cM^{-1}g(t) t^{\half - s} \, \frac{dt}{t}
    \\
    &= \thalf g(\thalf - s),
    \shortintertext{and}
    \int_0^\infty - \frac{1}{t^{\frac{3}{2}}} (\cM^{-1}g)' \!\left( \frac{1}{t} \right) t^s \, \frac{dt}{t}
    &= \int_0^\infty - (\cM^{-1}g)'(t) t^{\frac{3}{2} - s} \, \frac{dt}{t}
    \\
    &= (\tfrac{3}{2} - s) \int_0^\infty \cM^{-1}g(t) t^{\half - s} \, \frac{dt}{t}
    \\
    &= (\tfrac{3}{2} - s) g(\thalf - s)
    .
  \end{align*}

  Comparing the Mellin transforms of \eqref{eq:mellin_1} and \eqref{eq:mellin_2} gives
  \begin{align}
    \nonumber
    &F(\thalf - s)g(\thalf - s) \sim (2-s)g(\thalf - s)H_q(s)
    \\
    \label{eq:mellin_3}
    \Longleftrightarrow\quad
    &H_q(s) \sim \frac{F(\thalf - s)}{2-s}
    .
  \end{align}

  Now we express Sarnak's murmuration density in terms of $H_q$. In the sense of distributions,
  \begin{align*}
    \frac{1}{\#\cF(q)} \sum_{\rep \in \cF(q)} \frac{1}{\Delta y} \starsum*_{y < \frac{n}{q} < y + \Delta y} a_\rep(n) \sqrt{n}
    &= \frac{(y + \Delta y) \cM^{-1}H_q(y + \Delta y) - y \cM^{-1}H_q(y)}{\Delta y}
    \\
    &\sim \frac{d}{dy}y\cM^{-1}H_q(y)
  \end{align*}
  as $q \to \infty$ and $\Delta y \to 0$ appropriately.
  
  Substitute the straightforward calculation
  \begin{align*}
    \cM\!\left\{ \frac{d}{dy} y\cM^{-1}H_q(y) \right\} = -sH_q(s)
  \end{align*}
  into \eqref{eq:mellin_3} to obtain
  \begin{align*}
    \cM\!\left\{\frac{1}{\#\cF(q)} \sum_{\rep \in \cF(q)} \frac{1}{\Delta y} \starsum*_{y < \frac{n}{q} < y + \Delta y} a_\rep(n) \sqrt{n}\right\}
    &\sim -sH_q(s)
    \\
    &\sim \frac{s}{s-2}F(\thalf - s)
    .
    \qedhere
  \end{align*}
\end{proof}

\section{Quadratic characters}


Suppose that $2 < T < \x $ with $\x  \not\in \Z$, and that $\chi$ is a primitive nontrivial Dirichlet character modulo $d$. Following the proof of \cite[Thm.\ 12.10]{MV} yields
\begin{align}
  &\label{explicit_formula_dirichlet_1}
  \sum_{p^k < \x }
  \chi(p^k)\log p = -\frac{1}{2\pi i}\int_{\half + \eps - iT}^{\half + \eps + iT} \frac{L'(s,\chi)}{L(s,\chi)}\x^s\,\frac{ds}{s} + O\!\left(\x^{1+\eps} T^{-1+\eps} d^\eps\right).
\end{align}

Throughout this paper we adopt the convention that $p$ always denotes a prime and $p^k$ a prime power; the sum in \eqref{explicit_formula_dirichlet_1} is over $p$ prime and $k \in \Z_{>0}$.

\label{sec:kronecker_ratiosconjecture}

Define $\cF(D) \coloneqq \{d \,:\, 1 < d < D,\, \text{$d$ a fundamental discriminant}\}$. Let $\chi_d$ denote the Kronecker symbol $\left(\frac{d}{\cdot}\right)$, which, for $d \in \cF(D)$, is a real even primitive Dirichlet character modulo $d$.

The following theorem is an immediate consequence of the ratios conjecture \cite[Conj.\ 2.6]{conrey_snaith}. 
\begin{theorem}[Conrey--Snaith {\cite[Thm.\ 2.7]{conrey_snaith}}]\label{ratios_conjecture_dirichlet}
  Assume \cite[Conj.\ 2.6]{conrey_snaith}, that $\frac{1}{\log D} \ll \Re(r) < \tfrac{3}{4}$, and that $\Im(r) \ll D^{1-\eps}$. Then
  \begin{align*}
    \sum_{d \in \cF(D)} \frac{L'(\half + r, \chi_d)}{L(\half + r, \chi_d)} = &\,\,\sum_{d \in \cF(D)} \left[ \frac{\zeta'(1+2r)}{\zeta(1+2r)} + \sum_p \frac{\log p}{(p+1)(p^{1+2r}-1)} \right.\\
      &\left. - \left(\frac{d}{\pi}\right)^{\!-r} \frac{\Gamma\!\left(\frac{1}{4} - \frac{r}{2}\right)}{\Gamma\!\left(\frac{1}{4} + \frac{r}{2}\right)} \zeta(1 - 2r) \prod_p\left(1-\frac{1}{(p+1)p^{1-2r}}-\frac{1}{p+1}\right)\left(1-\frac{1}{p}\right)^{-1} \right] + \cR(D)
  \end{align*}
  with $\cR(D) \ll D^{\half + \eps}$.
\end{theorem}

\begin{remark} It was observed in \cite[Lemma 2.4]{miller08} and also pointed out to us by Mike Rubinstein that
  \begin{align*}
    \prod_p\left(1-\frac{1}{(p+1)p^{1-2r}}-\frac{1}{p+1}\right)\left(1-\frac{1}{p}\right)^{-1} = \frac{\zeta(2)}{\zeta(2-2r)}.
  \end{align*}
  
\end{remark}

\begin{remark}
  \label{rem:dirichet_functional_equation_factor}
  The factor of $\big(\tfrac{d}{\pi}\big)^{-r}$ in the last term of \cref{ratios_conjecture_dirichlet} is the one which will ultimately yield murmurations. It comes from the analytic conductor factor in the term of the approximate functional equation which uses the $L$-function's functional equation. Analogous remarks hold in every other example we investigate in this paper.
  See also \cref{sec:afe}.
\end{remark}



\begin{proof}[Proof of \cref{dirichlet_murmuration}]
Averaging \eqref{explicit_formula_dirichlet_1} over $\cF_\chi$ yields
\begin{align*}
  \frac{1}{\#\cF_\chi}&\sum_{d \in \cF_\chi}\sum_{p^k < \x } \chi_d(p^k)\log p = -\frac{1}{2\pi i}\int_{\half + \eps - iT}^{\half + \eps + iT} \frac{1}{\#\cF_\chi}\sum_{d \in \cF_\chi}\frac{L'(s,\chi_d)}{L(s,\chi_d)}\x^s\,\frac{ds}{s} + O\!\left(\x^{1+\eps} T^{-1+\eps} D^\eps\right).
\end{align*}

Assuming \cite[Conj.\ 2.6]{conrey_snaith}, \cref{ratios_conjecture_dirichlet} gives
\begin{align*}
  -\frac{1}{2\pi i}\int_{\half+\eps-iT}^{\half+\eps+iT} \frac{1}{\#\cF_\chi}\sum_{d \in \cF_\chi} \frac{L'(s,\chi_d)}{L(s,\chi_d)} \x^s\,\frac{ds}{s}
  ={}&-\frac{1}{2\pi i}\int_{\half+\eps-iT}^{\half+\eps+iT} \frac{\zeta'(2s)}{\zeta(2s)} \x^s\,\frac{ds}{s}\\
  &-\frac{1}{2\pi i}\int_{\half+\eps-iT}^{\half+\eps+iT} \sum_p \frac{\log p}{(p+1)(p^{2s}-1)} \x^s\,\frac{ds}{s}\\
  &+\frac{1}{2\pi i}\int_{\half+\eps-iT}^{\half+\eps+iT} \frac{\pi^2}{6}\frac{\Gamma\!\left(\frac{1-s}{2}\right)}{\Gamma\!\left(\frac{s}{2}\right)} \frac{\zeta(2-2s)}{\zeta(3-2s)} \frac{1}{\#\cF_\chi} \sum_{d \in \cF_\chi} \left(\frac{\pi}{d}\right)^{s - \half}  \x^s\,\frac{ds}{s}\\
  &-\frac{1}{2\pi i}\int_{\half+\eps-iT}^{\half+\eps+iT} \frac{\cR(D) - \cR(D_0)}{\#\cF_\chi} \x^s\,\frac{ds}{s}.
\end{align*}

From \cite[\S 12]{MV},
\begin{align}\label{zeta_explicit_formula}
  -\frac{1}{2\pi i}\int_{\half+\eps-iT}^{\half+\eps+iT} \frac{\zeta'(2s)}{\zeta(2s)} \x^s\,\frac{ds}{s} = \sum_{p^k < \x^\half} \log p + O\!\left(\x^{\half+\eps}T^{-1+\eps} + \x^\eps\right).
\end{align}

By inspection,
\begin{align}
  &-\frac{1}{2\pi i}\int_{\half+\eps-iT}^{\half+\eps+iT} \sum_p \frac{\log p}{(p+1)(p^{2s}-1)} \x^s\,\frac{ds}{s} \ll \x^\eps T^\eps + \x^{\half+\eps} T^{-1} \label{Aderiv_bound}\\
  \shortintertext{and}
  &-\frac{1}{2\pi i}\int_{\half+\eps-iT}^{\half+\eps+iT} \frac{\cR(D) - \cR(D_0)}{\#\cF_\chi} \x^s\,\frac{ds}{s} \ll \x^{\half+\eps} T^\eps \cR(D) \#\cF_\chi^{-1}.\nonumber
\end{align}

Collecting terms then yields \cref{dirichlet_murmuration}.
\end{proof}

\begin{remark}
  \label{rem:dirichlet_error}
  The error term of \cref{dirichlet_murmuration} comes more or less entirely from the ratios conjecture \cref{ratios_conjecture_dirichlet}: the term $O(\x^{\half + \eps} T^\eps \cR(D) \#\cF_\chi^{-1})$ explicitly so, and the term $O(\x^{1+\eps}T^{-1+\eps}D^\eps)$ essentially because of the restriction $T \ll D_0^{1 - \eps}$ in the hypothesis of \cref{ratios_conjecture_dirichlet}. Note that the restriction $T \ll \x^{1-\eps}$ coming from \eqref{explicit_formula_dirichlet_1} can easily be dropped if necessary, at the expense requiring that $\x $ not be taken very near a prime power. 
\end{remark}


Writing $\x  = Dy$ in \cref{dirichlet_murmuration} yields
\begin{align}\label{eq:murmuration_y}
  &\frac{1}{\#\cF_\chi}\sum_{d \in \cF_\chi}\frac{1}{\sqrt{D y}}\sum_{p < Dy} \chi_d(p)\log p
  \,\approx\,
  \frac{1}{2\pi i}\int_{\half+\eps-iT}^{\half+\eps+iT} \frac{\pi^2}{6}\frac{\Gamma\!\left(\frac{1-s}{2}\right)}{\Gamma\!\left(\frac{s}{2}\right)} \frac{\zeta(2-2s)}{\zeta(3-2s)}\left(\pi y\right)^{s - \half} \,\frac{ds}{s}.
\end{align}
The phenomenon of murmurations as discussed in \cite{LOP} and \cite{zubrilina} is analogous to the observation that the right hand side above doesn't depend on $\cF_\chi$. A more precise form of \eqref{eq:murmuration_y} could be formulated by emulating the style of \cref{elliptic_curve_murmuration_rho}.


\begin{remark}\label{difference_remark}
  By taking the difference of \cref{dirichlet_murmuration} evaluated at $\x $ and $\x  + \x^\delta$ one can isolate the contribution of primes in an interval of length $\x^\delta$. For any $\thalf < \delta < 1$ and any $\tfrac{3}{2} - \delta < \theta < 2\delta$, there exist $\alpha, \beta \in \R$ such that
  \begin{enumerate}
    \item $\x^\theta \ll D \ll \x^\theta,\,\, \x^\alpha \ll T \ll \x^\alpha,\,\, \x^\beta \ll \#\cF_\chi \ll \x^\beta$,
    \item $\alpha < 1$ and $\alpha < \theta$,
    \item $\beta < \theta$, and
    \item the ratio of the main term of \cref{dirichlet_murmuration} to the error term is $\gg \x^\eta$ for some $\eta > 0$ as $\x  \to \infty$.
  \end{enumerate}

  The last point above is saying more or less that, for any $\thalf < \delta < 1$ and $\tfrac{3}{2} - \delta < \theta < 2\delta$, \cref{dirichlet_murmuration} allows one to obtain a meaningful main term for murmurations for a range of primes of length $\x^\delta$ and roughly $\x^\beta$ discriminants of size roughly $\x^\theta$.

  The condition $\alpha < 1$ is included because our explicit formula \eqref{explicit_formula_dirichlet_1} requires $T < \x $ (an additional error term in Perron's formula appears when $T > \x $). 
  The condition $\alpha < \theta$ is included so that the condition $\Im(r) \ll D^{1-\eps}$ from \cref{ratios_conjecture_dirichlet} is satisfied. The condition $\beta < \theta$ is included because necessarily $\#\cF_\chi \ll D$. 
  It is plausible that more careful analysis could allow one to take smaller values of $\delta$.  
\end{remark}


It may be of interest to write the inverse Mellin transform on right hand side of \cref{dirichlet_murmuration} as a sum of residues. For $\thalf < c < \tfrac{3}{4}$ and $T > 4$,
define $\varphi$ and $\psi$ by
\begin{align}
  \label{def:phi}
  \varphi(y,s) &\coloneqq \frac{\pi^2}{6}\frac{\Gamma\!\left(\frac{1-s}{2}\right)}{\Gamma\!\left(\frac{s}{2}\right)} \frac{\zeta(2-2s)}{\zeta(3-2s)} y^{s - \half}\frac{1}{s}
  \shortintertext{and}
  \label{def:psi}
  \psi(y) &\coloneqq \frac{1}{2\pi i}\int_{c - iT}^{c + iT} \varphi(y,s)\,ds.
\end{align}

\begin{proposition}
  \label{prop:psi_as_residues}
  Assume the Riemann hypothesis for $\zeta$ and that all of $\zeta$'s zeros are simple. For $y < T$ and some $T < T_* < T+1$,
  \begin{align*}
    \psi(y)
    &= \sum_{\substack{|\gamma| < T_* \\ \zeta(\thalf + i\gamma) = 0}} \!\!\!\!\!\!\!\! -\frac{\pi^2}{12} \frac{\Gamma(-\tfrac{1}{8} + \tfrac{i\gamma}{4})}{\Gamma(\tfrac{5}{8} - \tfrac{i\gamma}{4})} \frac{\zeta(-\thalf + i\gamma)}{\zeta'(\thalf + i\gamma)} \frac{y^{\frac{3}{4} - \frac{i\gamma}{2}}}{\frac{5}{4} - \frac{i\gamma}{2}}
    + \sum_{k=0}^\infty -\frac{\pi^2}{12} \frac{\Gamma(-\tfrac{3}{4} - \tfrac{k}{2})}{\Gamma(\tfrac{5}{4} + \tfrac{k}{2})} \frac{\zeta(-3 - 2k)}{\zeta'(-2 - 2k)} \frac{y^{2 + k}}{\frac{5}{2} + k}
    + O\!\left(\left(\frac{y}{T}\right)^{\!\frac{3}{4}} T^{\eps}\right)
    \\
    &= \sum_{\substack{|\Im(w)| < T_* \\ \zeta(3-2w) = 0}} -\half \varphi(y,w) + O\!\left(\left(\frac{y}{T}\right)^{\!\frac{3}{4}} T^{\eps}\right)\!.
  \end{align*}
\end{proposition}

The assumption in \cref{prop:psi_as_residues} that the zeros of $\zeta$ are simple is made for convenience and can easily be dropped. We would expect that results similar to \cref{prop:psi_as_residues} hold for \cref{elliptic_curve_murmuration_rho,cech_murmuration,dirichlet_murmuration_GRH,newspace_murmuration_GRH,elliptic_curve_murmuration_GRH}.

The remainder of this section is dedicated to proving \cref{prop:psi_as_residues}. We begin by introducing \textit{Stirling's formula}. 

\cite[8.344]{GR} gives, for $|s| \to \infty$ with $-\pi < \arg(s) < \pi$ and bounded away from $\pm\pi$,
\begin{align}
  \label{eq:stirling_log}
  \log \Gamma(s) = (s - \thalf)\log s - s + \log\sqrt{2\pi} + O\!\left(\tfrac{1}{s}\right)
  .
\end{align}
The same reference gives more precise expansions too if necessary. Thus,
\begin{align}
  \label{eq:stirling_arg}
  \Gamma(s) = \sqrt{2\pi} \,s^{s - \half} e^{-s} \left(1 + O\!\left(\tfrac{1}{s}\right)\right)
\end{align}
and, for $s = \sigma + it$,
\begin{align}
  \label{eq:stirling}
  |\Gamma(s)| = \sqrt{2\pi} \,|s|^{\sigma - \half} e^{-\sigma} e^{-\arg(s) t} \left(1 + O\!\left(\tfrac{1}{s}\right)\right)
  .
\end{align}
In many contexts $\sigma$ is absolutely bounded and one is interested in the $t$ aspect of \eqref{eq:stirling} only. In such cases, the following slightly simpler approximation is handy: 
\begin{align}
  \label{eq:stirling_t}
  |\Gamma(\sigma + it)| = \sqrt{2\pi} \,|t|^{\sigma - \half} e^{-\frac{\pi}{2}|t|} \left(1 + O\!\left(\tfrac{1}{t}\right)\right)
  .
\end{align}

\begin{remark}
  \label{rem:stirling}
  We will often deal with products and ratios of gamma functions. Using only \eqref{eq:stirling}, one deduces e.g.\
  \begin{align}
    \nonumber
    |\Gamma(s_1)\Gamma(s_2)|
    =
    2\pi |s_1|^{\sigma_1 - \half} |s_2|^{\sigma_2 - \half} e^{-\sigma_1 - \sigma_2} e^{-\arg(s_1) t_1 - \arg(s_2)t_2}
    \left(1 + O\!\left(\frac{\Gamma(s_2)}{s_1} + \frac{\Gamma(s_1)}{s_2}\right) \right)
    .
  \end{align}
  However, returning to \eqref{eq:stirling_log}, we obtain the same approximation with a much better error term:
  \begin{align}
    \nonumber
    |\Gamma(s_1)\Gamma(s_2)|
    =
    2\pi |s_1|^{\sigma_1 - \half} |s_2|^{\sigma_2 - \half} e^{-\sigma_1 - \sigma_2} e^{-\arg(s_1) t_1 - \arg(s_2)t_2}
    \left(1 + O\!\left(\frac{1}{s_1} + \frac{1}{s_2}\right) \right)
    .
  \end{align}
\end{remark}


With Stirling's formula we bound the integrand $\varphi$ of the inverse Mellin transform $\psi$ which features in \cref{prop:psi_as_residues}.
\begin{lemma}
  \label{lemma:psi_bound}
  Assume the Riemann hypothesis for $\zeta$.
  Let $\varphi$ be as defined in \eqref{def:phi}.
  For $s = \sigma + it$ with $\sigma > \thalf$ and bounded away from integers and the lines $\Re(s) = \thalf, \tfrac{5}{4}$,
  \begin{align*}
    \varphi(y,s)
    &\ll
    \begin{cases}
      |s|^{-\sigma - \half + \eps} & \thalf+\eps < \sigma \leqslant \tfrac{3}{4}\\
      |s|^{\sigma - 2 + \eps} & \tfrac{3}{4} \leqslant \sigma < \tfrac{5}{4}-\eps\\
      |s|^{-\sigma + \half + \eps}y^{\sigma - \half} & \tfrac{5}{4}+\eps < \sigma.
    \end{cases}
  \end{align*}
\end{lemma}
\begin{proof}
  By \eqref{eq:stirling} we have
  \begin{align*}
    \frac{\Gamma\!\left(\frac{1-s}{2}\right)}{\Gamma\!\left(\frac{s}{2}\right)}
    &\ll |s|^{\half - \sigma}.
  \end{align*}

  For $\thalf+\eps < \sigma \leqslant \tfrac{3}{4}$, as a result of the Phragm\'en--Lindel\"of principle \cite[Thm.\ 8.2.1]{goldfeld:book}, the Lindel\"of hypothesis \cite[Thm.\ 13.18]{MV} (which follows from RH for $\zeta$), and \cite[Thm.\ 5.19]{IK}, we have
  \begin{align*}
    \frac{\zeta(2-2s)}{\zeta(3-2s)}
    \ll |s|^\eps.
  \end{align*}
  For $\tfrac{3}{4} \leqslant \sigma < \tfrac{5}{4}-\eps$, the functional equation and \eqref{eq:stirling} give
  \begin{align*}
    \frac{\zeta(2-2s)}{\zeta(3-2s)}
    &= \pi^{2s - \frac{3}{2}}\frac{\Gamma(s-\thalf)}{\Gamma(1-s)}\frac{\zeta(2s-1)}{\zeta(3-2s)}\\
    &\ll |s|^{2\sigma - \frac{3}{2} + \eps}.
  \end{align*}
  For $\tfrac{5}{4}+\eps < \sigma$,
  \begin{align*}
    \frac{\zeta(2-2s)}{\zeta(3-2s)}
    &= \pi^2\frac{\Gamma(s-\thalf)\Gamma(\tfrac{3}{2}-s)}{\Gamma(1-s)\Gamma(s-1)}\frac{\zeta(2s-1)}{\zeta(2s-2)}\\
    &\ll |s|^{1+\eps}.
  \end{align*}
  Combining these observations yields \cref{lemma:psi_bound}.
\end{proof}

For the remainder of this section we assume the Riemann hypothesis for $\zeta$ and that its zeros are simple. Recall that $\psi$ is defined in \eqref{def:psi}.

\begin{lemma}
  \label{lemma:psi_converges}
  $\psi(y)$ converges for all $y \in \R_{>0}$.
\end{lemma}
\begin{proof}
  Immediate consequence of \cref{lemma:psi_bound}.
\end{proof}

Our goal is to evaluate $\psi(y)$ by shifting the contour to the right. We will now bound the top, bottom, and right segments of the resulting rectangle.

\begin{lemma}
  \label{lemma:psi_vertical}
  \begin{align*}
    \frac{1}{2\pi i}\int_{M - iT}^{M + iT} \varphi(y,s)\,ds
    \ll \left(M\left(\frac{y}{M}\right)^{M-\half} + \frac{T}{M}\left(\frac{y}{T}\right)^{M-\half}\right) M^\eps T^\eps.
  \end{align*}
\end{lemma}
\begin{proof}
  Straightforward consequence of \cref{lemma:psi_bound}.
\end{proof}

\begin{lemma}
  \label{lemma:psi_horizontal}
  For $y < T$, there exists a path $\ell$ from $c + iT$ to $M + iT$ with $M > \tfrac{5}{4}$ such that
  \begin{align*}
    \frac{1}{2\pi i}\int_\ell \varphi(y,s)\,ds
    \ll \left(\frac{y}{T}\right)^{\frac{3}{4}} T^\eps,
  \end{align*}
  and idem for $\ell$ from $c - iT$ to $M - iT$.
\end{lemma}
\begin{proof}
  Let $\eps_0 < \eps$. We handle the integral in each of the following four regions separately:
  \begin{itemize}
  \item
    $c \leqslant \sigma < \tfrac{3}{4}$
  \item
    $\tfrac{3}{4} \leqslant \sigma < \tfrac{5}{4} - \eps_0$
  \item
    $\tfrac{5}{4} - \eps_0 \leqslant \sigma < \tfrac{5}{4} + \eps_0$
  \item
    $\tfrac{5}{4} + \eps_0 \leqslant \sigma$.
  \end{itemize}

  When $c \leqslant \sigma < \tfrac{3}{4}$, by \cref{lemma:psi_bound} we have
  \begin{align*}
    \int_{c + iT}^{\frac{3}{4} + iT} \varphi(y,s)\,ds
    &\ll \int_c^{\frac{3}{4}} |\sigma + iT|^{-\sigma - \half + \eps} d\sigma\\
    &\ll T^{-c - \half + \eps}.
  \end{align*}

  When $\tfrac{3}{4} \leqslant \sigma < \tfrac{5}{4} - \eps_0$,
  \begin{align*}
    \int_{\frac{3}{4} + iT}^{\frac{5}{4}-\eps_0 + iT} \varphi(y,s)\,ds
    &\ll \int_{\frac{3}{4}}^{\frac{5}{4}-\eps_0} |\sigma + iT|^{\sigma - 2 + \eps} d\sigma\\
    &\ll T^{-\frac{3}{4} + \eps}.
  \end{align*}

  Suppose $\tfrac{5}{4} - \eps_0 \leqslant \sigma < \tfrac{5}{4} + \eps_0$.
  By \cite[Thm.\ 10.13]{MV} there exists $0 < \delta < 1$ such that $|\sigma + i(T + \delta) - \rho| \gg \tfrac{1}{\log T}$ for all zeros $\rho$ of $\zeta$, where the implied constant is absolute. Choose $\delta$ as such. Then
  \begin{align*}
    \left(\int_{\frac{5}{4}-\eps_0 + iT}^{\frac{5}{4}-\eps_0 + i(T+\delta)} + \int_{\frac{5}{4}-\eps_0 + i(T+\delta)}^{\frac{5}{4}+\eps_0 + i(T+\delta)} + \int_{\frac{5}{4}+\eps_0 + i(T+\delta)}^{\frac{5}{4}+\eps_0 + iT} \right) \varphi(y,s)\,ds
    \ll T^{-\frac{3}{4} + \eps}.
  \end{align*}

  Finally, when $\tfrac{5}{4} + \eps_0 \leqslant \sigma$,
  \begin{align*}
    \int_{\frac{5}{4}+\eps_0 + iT}^{M + iT} \varphi(y,s)\,ds
    &\ll \int_{\frac{5}{4}+\eps_0}^M \left|\frac{y}{\sigma + iT}\right|^{\sigma - \half} |\sigma + iT|^\eps d\sigma\\
    &\ll \left(\frac{y}{T}\right)^{\frac{3}{4}} T^\eps.
  \end{align*}

  A path from $c - iT$ to $M - iT$ satisfying the lemma's conclusion is shown to exist in an identical manner.
\end{proof}

\begin{proof}[Proof of \cref{prop:psi_as_residues}]
  Consider the integral of $\varphi$ 
  along the contour
  \begin{itemize}
    \item
      from $c - iT$ to $c + iT$ in a straight line,
    \item
      then to $M + iT$ along a path satisfying the conclusion of \cref{lemma:psi_horizontal},
    \item
      then to $M - iT$ in a straight line,
    \item
      then to $c - iT$ along a path satisfying the conclusion of \cref{lemma:psi_horizontal}.
  \end{itemize}
  The result of \cref{prop:psi_as_residues} then follows from \cref{lemma:psi_converges,lemma:psi_vertical,lemma:psi_horizontal} and letting $M \to \infty$. The poles of $\varphi$ are at the zeros of $\zeta(3 - 2s)$. The convergence of the sum over the nontrivial zeros of $\zeta$ in \cref{prop:psi_as_residues} as $T \to \infty$ follows from the fact that that the integral along the contour described above converges. Note that this sum does not converge absolutely.
\end{proof}



\section{Elliptic curves ordered by height}\label{sec:elliptic_curves}

In this section we prove \cref{elliptic_curve_murmuration_rho}. In \cref{sec:ec_raw} we use the ratios conjecture \cite[Thm.\ 3.8]{DHP} to prove \cref{elliptic_curve_murmuration}, much like how 
\cite[Thm.\ 2.7]{conrey_snaith} led to \cref{dirichlet_murmuration}.

A characteristic feature of murmurations is their invariance under a simultaneous scaling of the range of primes being averaged and the range of analytic conductors. For example, this property is highlighted in \eqref{eq:murmuration_y}, which follows straightforwardly from \cref{dirichlet_murmuration}. In contrast, \cref{elliptic_curve_murmuration} does not clearly exhibit this sort of ``$p/N$-invariance'', as it is unclear how the distribution of conductors $N_E$ in the family $\cF(H)$ behaves as $H$ grows. To this end, in \cite{conductor} we prove that, for any $\lambda_1 > \lambda_0 > \tfrac{4464}{\log H}$,
\begin{align}\label{conductor_distribution}
  &\frac{\#\!\left\{ E \in \cF(H) \,:\, \lambda_0 < \frac{N_E}{H} < \lambda_1 \right\}}{\#\cF(H)} = F_N(\lambda_1) - F_N(\lambda_0) + O( (\log H)^{-1+\eps})
\end{align}
and
\begin{align}\label{small_conductors}
  &\frac{\#\!\left\{E \in \cF(H) \,:\, \frac{N_E}{H} < \lambda_0 \right\}}{\#\cF(H)} \ll \lambda_0^{\frac{5}{6}},
\end{align}
where $F_N$ is as in \cref{FN_def}. We plot $F_N$ and its derivative in \cite{conductor}.

In \cref{sec:ec_dist} we prove \cref{elliptic_curve_murmuration_rho}, which follows from \cref{elliptic_curve_murmuration} analogously to how \eqref{eq:murmuration_y} follows from \cref{dirichlet_murmuration}, and which makes clear the presence of the characteristic $p/N$-invariance. Our formulation here can also be applied to \eqref{eq:murmuration_y} to obtain a more precise statement.

\subsection{Ratios conjecture}\label{sec:ec_raw}

Let $E_{a,b}$ denote the elliptic curve with Weierstrass equation $y^2 = x^3 + ax + b$. Fix $H > 0$ and $r,t \in \Z$ such that $3\nmid r$ and $2\nmid t$. The following family of elliptic curves ordered by height appeared in \cite{young:elliptic_curves} and subsequently \cite{DHP}.
\begin{definition}\label{cF_def}
\begin{align*}
  \cF(H) \coloneqq \left\{E_{a,b} \,:\, a = r\mod 6,\, b = t\mod 6,\, |a| \leqslant H^{\frac{1}{3}},\, |b| \leqslant H^{\half},\, p^4 \mid a \implies p^6 \nmid b\right\}.
\end{align*}
\end{definition}
The conditions on $a$ and $b$ above ensure that no two curves in $\cF(H)$ are isomorphic and that the model $E_{a,b} \,:\, y^2 = x^3 + ax + b$ is minimal at every prime \cite[\S 2.1]{young:elliptic_curves}.

Given an elliptic curve $E$, let $N_E$ denote its conductor and $\rootnum_E$ its root number. We define the $L$-function attached to $E$ via the Euler product
\begin{align*}
  L(s,E) \coloneqq \prod_{p\mid N_E}\left(1 - \frac{a_p(E)}{p^{s+\half}}\right)^{-1} \prod_{p\nmid N_E}\left(1 - \frac{\alpha_p}{p^s}\right)^{-1}\left(1 - \frac{\overline{\alpha}_p}{p^s}\right)^{-1}
\end{align*}
(where $\overline{\alpha}_p$ is the complex conjugate of $\alpha_p$) with $|\alpha_p| = |\overline{\alpha}_p| = 1$ and $\alpha_p + \overline{\alpha}_p = \frac{a_p(E)}{p^\half}$. With this normalization, the functional equation for $L(s,E)$ is given by \cite[\S 2.2]{young:elliptic_curves}
\begin{align*}
  \Lambda(s,E) &\coloneqq \left(\frac{\sqrt{N_E}}{2\pi}\right)^{s+\half}\Gamma(s + \thalf)L(s,E)\\
  &= \rootnum_E\Lambda(1-s,E).
\end{align*}

Let $\alpha$ and $\gamma$ be complex numbers satisfying $\Re(\alpha), \Re(\gamma) > 0$, and let $\mathrm{Tr}_k(p)$ denote the trace of the Hecke operator at $p$ on the space of level $1$ weight $k$ holomorphic cusp forms. 
\begin{definition}[{\cite[(3.38)]{DHP}}]\label{A_def}
\begin{align*}
  A(\alpha,\gamma) \coloneqq \frac{\zeta(1+\alpha+\gamma)}{\zeta(1+2\gamma)}\prod_{p=2,3} \frac{1 - a_p(E_{r,t})p^{-1-\gamma} + p^{-1-2\gamma}}{1 - a_p(E_{r,t})p^{-1-\alpha} + p^{-1-2\alpha}} \cdot \prod_{p > 3} \left[1 + \left(1 - \frac{p^9 - 1}{p^{10} - 1}\right)\left(\frac{1}{p^{1+2\gamma}} - \frac{1}{p^{1 + \alpha + \gamma}}\right.\right.&\\
    + \left.\left.\frac{p^{-2-\alpha-\gamma} - p^{-2-2\gamma}}{p^{2+2\alpha}-1} + \frac{p^{1+2\alpha+\gamma} - p^{1+\alpha+2\gamma} + p^\gamma - p^\alpha}{p^{\frac{3}{2} + \alpha + 2\gamma}} \sum_{m=5}^\infty \frac{\mathrm{Tr}_{2m+2}(p)}{p^{2m(\alpha+1) + \half}}\right)\right].&
\end{align*}
\end{definition}

The function $A(\alpha,\gamma)$ extends to a nonzero holomorphic function on $\Re(\alpha), \Re(\gamma) > -\tfrac{1}{4}$, and $A(r,r) = 1$ in this region \cite[Thm.\ 3.2]{DHP}. Henceforth when we refer to the function $A$ we mean this holomorphic extension.

The factors at $p = 2,3$ in the definition of $A(\alpha,\gamma)$ depend on $a_p(E_{r,t})$. For convenience we give these values in \cref{table_a2a3}.
%
\begin{table}[H]
\begin{tabular}{|c|c|c||c|c|c|}
  \hline
  $r$, $t$ & $a_2(E_{r,t})$ & $a_3(E_{r,t})$ & $r$, $t$ & $a_2(E_{r,t})$ & $a_3(E_{r,t})$ \rule{0pt}{1em}\\
  [0.2ex]
  \hline
  $1,1$ & $0$ & $0$  & $2,1$ & $0$ & $-3$ \rule{0pt}{1em}\\
  $1,3$ & $0$ & $0$  & $2,3$ & $0$ & $0$ \rule{0pt}{1em}\\
  $1,5$ & $0$ &  $0$ & $2,5$ & $0$ & $3$ \rule{0pt}{1em}\\
  \hline
  $4,t$ & \multicolumn{2}{c||}{Same as $1,t$} & $5,t$ & \multicolumn{2}{c|}{Same as $2,t$} \rule{0pt}{1em}\\
  \hline
\end{tabular}
\begin{tablecap}\label{table_a2a3}
  Values of $a_p(E_{r,t})$ for $p = 2,3$ and $r,t$ as in the definition of $\cF(H)$.
\end{tablecap}
\vspace{-\baselineskip}
\end{table}

The function $A_\alpha(r,r)$ is defined \cite[(3.42)]{DHP}
\begin{align*}
  A_\alpha(r,r) \coloneqq \left.\frac{\dee}{\dee \alpha}A(\alpha,\gamma)\right|_{\alpha = \gamma = r}.
\end{align*}
Note that $A(\alpha,\gamma)$ and $A_\alpha(r,r)$ are bounded in the regions $\Re(\alpha), \Re(\gamma) > -\tfrac{1}{4}$ and $\Re(r) > -\tfrac{1}{4}$ respectively.

An analogue of \cite[Thm.\ 2.7]{conrey_snaith} for $\cF(H)$ is given in \cite[Thm.\ 3.8]{DHP}.
\begin{theorem}[David--Huynh--Parks {\cite[Thm.\ 3.8]{DHP}}]\label{ratios_conjecture_elliptic_curves}
  Assume \cite[Conj.\ 3.7]{DHP}, that $\Re(r) \gg \frac{1}{\log H}$ and $\Im(r) \ll H^{1-\eps}$. With the notation above,
  \begin{align*}
    \sum_{E \in \cF(H)} \frac{L'(\half + r, E)}{L(\half + r, E)} = &\,\,\sum_{E \in \cF(H)}\left[-\frac{\zeta'(1+2r)}{\zeta(1+2r)} + A_\alpha(r,r)\right.\\
    &\left. -\rootnum_E\left(\frac{N_E}{4\pi^2}\right)^{-r}\frac{\Gamma(1-r)}{\Gamma(1+r)}\zeta(1+2r)A(-r,r)\right] + \cR(H)
  \end{align*}
  with $\cR(H) \ll H^{\frac{1}{3} + \eps}$.
\end{theorem}
Discussion of the error term $\cR(H)$ can be found immediately below the statement of \cite[Conj.\ 3.7]{DHP}. We would have guessed that $\cR(H) \gg \#\cF(H)^{\half+\eps} \gg H^{\frac{5}{12} + \eps}$.

Suppose that $2 < T < \x $ and $\x  \not\in \Z$. The explicit formula \eqref{explicit_formula_elliptic_curve_1} follows from the proof of \cite[Lemma 2.1]{fiorilli}.
\begin{align}\label{explicit_formula_elliptic_curve_1}
  \sum_{p^k < \x , \, p\,\nmid N_E}
  \left(\alpha_p^k + \overline{\alpha}_p^k\right)\log p = -\frac{1}{2\pi i}\int_{1 + \eps - iT}^{1 + \eps + iT} \frac{L'(s,E)}{L(s,E)} \x^s\,\frac{ds}{s} + O\!\left(\x^{1+\eps} T^{-1+\eps} N_E^\eps\right).
\end{align}

\Cref{elliptic_curve_murmuration} below is the elliptic curve analogue of \cref{dirichlet_murmuration}.
\begin{proposition}\label{elliptic_curve_murmuration}
  Assume \cite[Conj.\ 3.7]{DHP}. With the notation above,
  \begin{align*}
    \frac{1}{\#\cF(H)}&\frac{1}{\sqrt{\x}}\sum_{E \in \cF(H)}\sum_{\substack{p^k < \x  \\ p\,\nmid N_E}} \left(\alpha_p^k + \overline{\alpha}_p^k\right)\log p\\
    = &\sum_{\rootnum = \pm 1} \frac{\rootnum}{2\pi i} \int_{\half + \eps - iT}^{\half + \eps + iT} \frac{\Gamma(\frac{3}{2} - s)}{\Gamma(\half + s)} \zeta(2s) A(\thalf - s, s - \thalf) \frac{1}{\#\cF(H)}\sum_{\substack{E\in\cF(H) \\ \rootnum_E = \rootnum}} \left(\frac{4\pi^2 \x }{N_E}\right)^{s-\half}\,\frac{ds}{s}\\
    &- \frac{1}{\sqrt{\x}}\sum_{p^k < \x^\half} \log p  + O\!\left(\x^{\eps}T^\eps\cR(H)\#\cF(H)^{-1} + \x^{\half+\eps} T^{-1 + \eps} H^\eps \right).
  \end{align*}
\end{proposition}
\begin{proof}
  Essentially identical to the proof of \cref{dirichlet_murmuration}.
\end{proof}

\begin{remark}
  \label{rem:eq:elliptic_curve_conductor_guess}
In light of \cref{elliptic_curve_murmuration}, we wonder if, for e.g.\ $\cF_N(C)^\rootnum \coloneqq \{E \,:\, C < N_E < C + C^\delta,\,\, \rootnum_E = \rootnum\}$ with $\delta > \thalf$, 
\begin{align}
  \nonumber
  \frac{1}{\#\cF_N(C)^\rootnum}&\frac{1}{\sqrt{\x}}\sum_{E \in \cF_N(C)^\rootnum}\sum_{\substack{p^k < \x  \\ p\,\nmid N_E}} \left(\alpha_p^k + \overline{\alpha}_p^k\right)\log p\\
  ={}&\frac{\rootnum}{2\pi i} \int_{\half + \eps - iT}^{\half + \eps + iT} \frac{\Gamma(\frac{3}{2} - s)}{\Gamma(\half + s)} \zeta(2s) A(\thalf - s, s - \thalf) \frac{1}{\#\cF_N(C)^\rootnum}\sum_{E\in\cF_N(C)^\rootnum} \left(\frac{4\pi^2 \x }{N_E}\right)^{s-\half}\,\frac{ds}{s}
  \nonumber
  \\
  &- \frac{1}{\sqrt{\x}}\sum_{p^k < \x^\half} \log p  + O\!\left(\x^{\eps}T^\eps(\#\cF_N(C)^\rootnum)^{-\half + \eps} + \x^{\half+\eps} T^{-1 + \eps} H^\eps \right).\nonumber
\end{align}
\end{remark}

\subsection{Proof of \cref{elliptic_curve_murmuration_rho}}\label{sec:ec_dist}

Elliptic curves are most naturally ordered by conductor --- e.g. this is the quantity which fundamentally governs the behaviour of murmurations --- but most easily ordered by height --- this is the ordering used in \cite[Conj.\ 3.7]{DHP}. Converting between these two orderings is an interesting and difficult problem \cite{watkins, ssw, cremona_sadek}, and one essential for demonstrating that \cref{elliptic_curve_murmuration} leads to murmurations. We establish some notation so that we may refer to the results of \cite{conductor}, which will then allow us to prove \cref{elliptic_curve_murmuration_rho}.

\begin{definition}\label{FDelta_def}
  $$F_\Delta(\lambda) \coloneqq \frac{1}{4}\int_{-1}^1\int_{-1}^1 \begin{cases}1 & \text{if $-16(4\alpha^3 + 27\beta^2) < \lambda$} \\ 0 & \text{otherwise}\end{cases} \,d\alpha \,d\beta$$
\end{definition}

\begin{definition}\label{rho_def}
  For any prime $p$ and any integer $m$, define $\rho(p,m)$ to be the corresponding entry in \cref{table_rho_def}.\\
  
\centering{
  \begin{tabular}{|l||l|l|l|}
  \hline
  $\gcd(m, p^\infty)$ & $p \geqslant5$ & $p = 2$ & $p = 3$ \rule{0pt}{1.4em}\\
  [1.1ex]
  \hline
  \hline
  $p^0$ & $1 - \frac{1}{p^2}$ & $\half$ & $1$ \rule{0pt}{1.4em}\\ [0.2ex]
  $p^1$ & $\frac{1}{p^2}\!\left(1-\frac{1}{p}\right)$ & $\frac{1}{4}$ &  \\ [0.2ex]
  $p^2$ & $\frac{1}{p^3}\!\left(1-\frac{1}{p}\right)$ & $\frac{1}{4}$ &  \\ [0.2ex]
  $p^3$ & $\frac{1}{p^4}\!\left(1-\frac{1}{p}\right)\!\left(1 - \frac{1}{p}\right)$ &  &  \\ [0.2ex]
  $p^4$ & $\frac{1}{p^5}\!\left(1-\frac{1}{p}\right)\!\left(2 - \frac{1}{p}\right)$ &  &  \\ [0.2ex]
  $p^5$ & $\frac{1}{p^6}\!\left(1-\frac{1}{p}\right)\!\left(2 - \frac{2}{p}\right)$ &  &  \\ [0.2ex]
  $p^6$ & $\frac{1}{p^7}\!\left(1-\frac{1}{p}\right)\!\left(3 - \frac{2}{p}\right)$ &  &  \\ [0.2ex]
  $p^7$ & $\frac{1}{p^8}\!\left(1-\frac{1}{p}\right)\!\left(3 - \frac{2}{p}\right)$ &  &  \\ [0.2ex]
  $p^8$ & $\frac{1}{p^9}\!\left(1-\frac{1}{p}\right)\!\left(3 - \frac{2}{p}\right)$ &  &  \\ [0.2ex]
  $p^n$, $n \geqslant9$ & $\frac{1}{p^{n+1}}\!\left(1-\frac{1}{p}\right)\!\left(2 - \frac{2}{p}\right)$ &  &  \\ [0.2ex]
  \hline
\end{tabular}
}
\begin{tablecap}\label{table_rho_def}
  Values of $\rho(p,m)$. Blank entries denote a value of $0$.
\end{tablecap}
\end{definition}

\begin{definition}\label{FN_def}
  Let $F_\Delta$ and $\rho$ be as in \cref{FDelta_def,rho_def}.
  For $\lambda > 0$, define
  \begin{align*}
    F_N(\lambda) \coloneqq \frac{\zeta^{(6)}(10)}{\zeta^{(6)}(2)} \sum_{m = 1}^\infty \big(F_\Delta(m\lambda) - F_\Delta(-m\lambda)\big)\cdot\rho(2,m)\rho(3,m)\prod_{\substack{p \geqslant5 \\ p \mid m}} \frac{\rho(p,m)}{1 - p^{-2}},
  \end{align*}
  and set $F_N(\lambda) \coloneqq 0$ when $\lambda \leqslant 0$.
\end{definition}

\begin{remark}
  \label{rem:FN_computable}
  To compute values of the function $F_N(\lambda)$, one may use the identity
  \begin{align*}
    &\frac{\zeta^{(6)}(10)}{\zeta^{(6)}(2)} \sum_{m = 1}^\infty \big(F_\Delta(m\lambda) - F_\Delta(-m\lambda)\big)\cdot\rho(2,m)\rho(3,m)\prod_{p\geqslant 5} \frac{\rho(p,m)}{1 - p^{-2}}
    \\&
    \hspace{4cm}
    = 1 + \frac{\zeta^{(6)}(10)}{\zeta^{(6)}(2)} \sum_{1 \leqslant m \leqslant \frac{496}{\lambda_0}} \big(F_\Delta(m\lambda) - F_\Delta(-m\lambda) - 1\big)\cdot\rho(2,m)\rho(3,m)\prod_{p \geqslant 5} \frac{\rho(p,m)}{1 - p^{-2}}
    ,
  \end{align*}
  valid for any $\lambda_0 \leqslant \lambda$ (choosing $\lambda_0 \approx \lambda_1$ is most convenient). The product over primes is finite, as all but finitely many factors are $1$. See \cite{conductor} for details. The code used to plot the function $F_N(\lambda)$ and its derivative in \cite{conductor} is available at \cite{github_conductors}.
  The derivative $F_N'(\lambda)$ is more straightforward to compute, since $F_N(\lambda_1) - F_N(\lambda_0)$ is a finite sum of a finite product for all $\lambda_1 > \lambda_0 > 0$.
\end{remark}

\begin{proof}[Proof of \cref{elliptic_curve_murmuration_rho}]
  Take $\lambda_0 > \tfrac{4464}{\log H}$.
  We split the integral in main term of \cref{elliptic_curve_murmuration} into two parts:
  \begin{align}
    &\int_{\half + \eps - iT}^{\half + \eps + iT} \frac{\Gamma(\frac{3}{2} - s)}{\Gamma(\half + s)} \zeta(2s) A(\thalf - s, s - \thalf) \frac{1}{\#\cF(H)^\rootnum}\sum_{E\in\cF(H)^\rootnum} \left(\frac{4\pi^2 \x }{N_E}\right)^{s-\half}\,\frac{ds}{s} \nonumber\\
    &\hspace{2cm}= \int_{\half + \eps - iT}^{\half + \eps + iT} \frac{\Gamma(\frac{3}{2} - s)}{\Gamma(\half + s)} \zeta(2s) A(\thalf - s, s - \thalf) \frac{1}{\#\cF(H)^\rootnum}\sum_{\substack{E\in\cF(H)^\rootnum \\ N_E < \lambda_0 H}} \left(\frac{4\pi^2 \x }{N_E}\right)^{s-\half}\,\frac{ds}{s} \label{main_term_small}\\
    &\hspace{2cm}+ \int_{\half + \eps - iT}^{\half + \eps + iT} \frac{\Gamma(\frac{3}{2} - s)}{\Gamma(\half + s)} \zeta(2s) A(\thalf - s, s - \thalf) \frac{1}{\#\cF(H)^\rootnum}\sum_{\substack{E\in\cF(H)^\rootnum \\ N_E > \lambda_0 H}} \left(\frac{4\pi^2 \x }{N_E}\right)^{s-\half}\,\frac{ds}{s}. \label{main_term_large}
  \end{align}
  
  We first consider \eqref{main_term_small}.
  \begin{align*}
    &\int_{\half + \eps - iT}^{\half + \eps + iT} \frac{\Gamma(\frac{3}{2} - s)}{\Gamma(\half + s)} \zeta(2s) A(\thalf - s, s - \thalf) \frac{1}{\#\cF(H)^\rootnum}\sum_{\substack{E\in\cF(H)^\rootnum \\ N_E < \lambda_0 H}} \left(\frac{4\pi^2 \x }{N_E}\right)^{s-\half}\,\frac{ds}{s}\\
    &\hspace{2cm} = \frac{1}{\#\cF(H)^\rootnum}\sum_{\substack{E\in\cF(H)^\rootnum \\ N_E < \lambda_0 H}} \int_{\half + \eps - iT}^{\half + \eps + iT} \frac{\Gamma(\frac{3}{2} - s)}{\Gamma(\half + s)} \zeta(2s) A(\thalf - s, s - \thalf)  \left(\frac{4\pi^2 \x }{N_E}\right)^{s-\half}\,\frac{ds}{s}.
  \end{align*}
  Using \eqref{small_conductors} with $\cF(H)$ replaced by $\cF(H)^\rootnum$, and using the fact that $N_E \geqslant1$,
  \begin{align*}
    &\frac{1}{\#\cF(H)^\rootnum}\sum_{\substack{E\in\cF(H)^\rootnum \\ N_E < \lambda_0 H}} \int_{\half + \eps - iT}^{\half + \eps + iT} \frac{\Gamma(\frac{3}{2} - s)}{\Gamma(\half + s)} \zeta(2s) A(\thalf - s, s - \thalf)  \left(\frac{4\pi^2 \x }{N_E}\right)^{s-\half}\,\frac{ds}{s}\\
    &\hspace{4cm}\ll \lambda_0^{\frac{5}{6}} \int_{\half + \eps - iT}^{\half + \eps + iT} \frac{\Gamma(\frac{3}{2} - s)}{\Gamma(\half + s)} \zeta(2s) A(\thalf - s, s - \thalf)  \x^{s-\half}\,\frac{ds}{s}.
  \end{align*}
  By Stirling's formula \cite[8.328.1]{GR}, the integrand on the right decays as
  \begin{align*}
    \frac{\Gamma(\frac{3}{2} - s)}{\Gamma(\half + s)} \zeta(2s) A(\thalf - s, s - \thalf)  \x^{s-\half}\,\frac{1}{s} \,\ll\, \left(\frac{\x }{|t|^2}\right)^{\sigma - \half} |t|^{-1}.
  \end{align*}
  Shift the contour to $\Re(s) = \tfrac{1}{4}+\eps$. The only pole in the region is at $s = \thalf$ and has residue $\lambda_0^{\frac{5}{6}}$. The integral along the left segment is $\ll \lambda_0^{\frac{5}{6}} (\tfrac{\x }{T^2})^{-\frac{1}{4} + \eps}$, and the integrals along the time and bottom segments are $\ll \lambda_0^{\frac{5}{6}} T^{-1}(\tfrac{\x }{T^2})^\eps$. Taking $T \gg \x^{\half + \eps}$,
  \begin{align}\label{main_term_small_bound}
    \text{\eqref{main_term_small}} \ll \lambda_0^{\frac{5}{6}}.
  \end{align}

  We now turn our attention to \eqref{main_term_large}, which will be the source of the main term of \cref{elliptic_curve_murmuration_rho}. Define
  \begin{align*}
    \tilde{F}(\lambda) \coloneqq \frac{\#\left\{E \in \cF(H)^\rootnum \,:\, \lambda_0 < \frac{N_E}{H} \leqslant \lambda \right\}}{\#\cF(H)^\rootnum}.
  \end{align*}
  Then, for any holomorphic function $f$, applying \cite[Exercise 1.7.3 (ii)]{durrett} to each of the real and imaginary parts of $f$ separately,
  \begin{align}\label{integration_by_parts}
    \frac{1}{\#\cF(H)^\rootnum} \sum_{\substack{E \in \cF(H)^\rootnum \\ a < \frac{N_E}{H} \leqslant b}} f\!\left(\frac{N_E}{H}\right) = f(b)\tilde{F}(b) - f(a)\tilde{F}(a) - \int_a^b \tilde{F}(\lambda)f'(\lambda) \,d\lambda
  \end{align}
  This is an instance of integration by parts for Riemann--Stieljes integrals.
  
  Taking $$f(\lambda) \coloneqq \left(\frac{4\pi^2 \x }{\lambda H}\right)^{s - \half}$$ for fixed $s \in \CC$, we get
  \begin{align*}
    &\int_{\half + \eps - iT}^{\half + \eps + iT} \frac{\Gamma(\frac{3}{2} - s)}{\Gamma(\half + s)} \zeta(2s) A(\thalf - s, s - \thalf) \frac{1}{\#\cF(H)^\rootnum}\sum_{\substack{E\in\cF(H)^\rootnum \\ N_E > \lambda_0 H}} \left(\frac{4\pi^2 \x }{N_E}\right)^{s-\half}\,\frac{ds}{s} \\
    &\hspace{2cm}=  \int_{\half + \eps - iT}^{\half + \eps + iT} \frac{\Gamma(\frac{3}{2} - s)}{\Gamma(\half + s)} \zeta(2s) A(\thalf - s, s - \thalf) \int_{\lambda_0}^\infty (s - \thalf)\left(\frac{4\pi^2 \x }{\lambda H}\right)^{s-\half} \tilde{F}(\lambda)\,\frac{d\lambda}{\lambda} \,\frac{ds}{s}. 
  \end{align*}
  Using \eqref{conductor_distribution} with $\cF(H)$ replaced by $\cF(H)^\rootnum$,
  \begin{align}
    &
    \int_{\lambda_0}^\infty (s - \thalf)\left(\frac{4\pi^2 \x }{\lambda H}\right)^{s-\half} \tilde{F}(\lambda)\,\frac{d\lambda}{\lambda} \nonumber\\
    &\hspace{2cm}
    = \int_{\lambda_0}^\infty (s - \thalf)\left(\frac{4\pi^2 \x }{\lambda H}\right)^{s-\half} \left(F_N(\lambda) - F_N(\lambda_0) + O((\log H)^{-1+\eps})\right)\frac{d\lambda}{\lambda}.\label{big_lambda_Fhat}
  \end{align}
  Using the fact that $F_N(\lambda) \ll \lambda^{\frac{5}{6}}$ \cite[Lemma 6.6]{conductor} and reasoning similarly to the justification of \eqref{main_term_small_bound}, we see that the error term in \eqref{big_lambda_Fhat} is negligible.
  Using integration by parts on the main term,
  \begin{align*}
    \int_{\lambda_0}^\infty (s - \thalf)\left(\frac{4\pi^2 \x }{\lambda H}\right)^{s-\half} \left(F_N(\lambda) - F_N(\lambda_0)\right)\frac{d\lambda}{\lambda}
    = \int_{\lambda_0}^\infty \left(\frac{4\pi^2 \x }{\lambda H}\right)^{s-\half} F_N'(\lambda)\,d\lambda.
  \end{align*}
  Including the remaining terms and coefficients from \cref{elliptic_curve_murmuration}, taking $\lambda_0 \ll (\log H)^{-1}$, and writing $\x  = Hy$ then yields \cref{elliptic_curve_murmuration_rho}.
\end{proof}


\section{Other examples}
\label{sec:other_examples}

\subsection{Quadratic characters --- Ratios conjecture, smoothed}
\label{sec:kronecker_2}

Here we present work of Miller \cite{miller08} on the ratios conjecture for quadratic Dirichlet characters.
Let $g$ be an even Schwartz function with the property that its Fourier transform
$$\hat{g}(\xi) \coloneqq \frac{1}{2\pi}\int_\R g(x)e^{-2\pi i\xi x}\, dx$$
is supported on $(-\sigma, \sigma)$. Define
\begin{align*}
  &\cF_1(D) \coloneqq \left\{ d \,:\, 0 < d < D,\, \text{$d$ a fundamental discriminant} \right\}\\
  &\cF_2(D) \coloneqq \left\{ 8d \,:\, 0 < d < \tfrac{D}{8},\, \text{$d$ an odd positive squarefree integer} \right\}.
\end{align*}

Combining theorems 1.1, 1.2, and 1.5 of \cite{miller08}, one obtains the following.
\begin{theorem}[{based on Miller \cite{miller08}}]\label{dirichlet_murmuration_GRH}
  Assume GRH. With the notation above and $j \in \{1,2\}$,
  \begin{align*}
    \frac{1}{\#\cF_j(D)} \sum_{d \in \cF_j(D)} &\sum_{p} \sum_{k=1}^\infty \hat{g}\!\left(\frac{\log p^k}{\log D}\right) \frac{\chi_d(p)^k \log p}{\sqrt{p^k}}\\
    = &\,\,\int_0^\infty \hat{g}\left(1+\frac{\log y}{\log D}\right) \int_{-\infty}^\infty \frac{\pi^2}{6}\frac{\Gamma(\frac{1}{4} - \pi it)}{\Gamma(\frac{1}{4} + \pi it)} \frac{\zeta(1 - 4\pi it)}{\zeta(2 - 4\pi it)} \frac{1}{\#\cF_j(D)} \sum_{d \in \cF_j(D)} \left(\frac{\pi Dy}{d}\right)^{2\pi it} dt\, \frac{dy}{y}\\
    &- \log D \int_{-\infty}^\infty g(t\log D) \left( \frac{\zeta'}{\zeta}(1 + 4\pi it) + \sum_{p} \frac{\log p}{(p+1)(p^{1 + 4\pi it} - 1)} \right) dt\\
    &+ O\!\left(D^{-\half + \eps} + \begin{cases}D^{-\frac{1 - \sigma}{2} + \eps} & \text{if $j = 1$} \\ D^{-\frac{2 - 3\sigma}{2} + \eps} + D^{-\frac{3 - 3\sigma}{4} + \eps} & \text{if $j = 2$} \end{cases} \right).
  \end{align*}
\end{theorem}

The family $\cF_1(D)$ was studied in \cref{sec:kronecker_ratiosconjecture}. There the main result was \cref{dirichlet_murmuration}, which can be compared to the above.

We can present \cref{dirichlet_murmuration_GRH} in a manner parallel to \cref{elliptic_curve_murmuration_rho}.
It is shown in \cite[Lemma B.1]{miller08} that $\#\cF_1(D) = \tfrac{3}{\pi^2}D \,+\, O(D^\half)$. Similarly, $\#\cF_2(D) = \tfrac{1}{2\pi^2}D \,+\, O(D^\half)$. Emulating the style of \cref{elliptic_curve_murmuration_rho}, the murmuration term of \cref{dirichlet_murmuration_GRH} becomes
\begin{align}\label{kronecker_GRH_main_term}
  \int_0^1 \int_0^\infty \hat{g}\!\left(1 + \frac{\log y}{\log D}\right) \int_{-\infty}^\infty \frac{\pi^2}{6}\frac{\Gamma(\frac{1}{4} - \pi it)}{\Gamma(\frac{1}{4} + \pi it)} \frac{\zeta(1 - 4\pi it)}{\zeta(2 - 4\pi it)} \left(\frac{\pi y}{\lambda}\right)^{2\pi it} dt\, \frac{dy}{y} \, d\lambda,
\end{align}

The error term of \cref{dirichlet_murmuration_GRH} is small only for $\sigma < 1$, which precludes taking $\hat{g}(\xi)$ to be sharply peaked near $\xi = 1$. Moreover, when $\sigma < 1$ the main term is asymptotically constant \cite[(1.6)]{miller08}. The restriction which causes $p/N \approx 1$ to be just outside the range admissible in \cref{dirichlet_murmuration_GRH} comes from two sources, one of which is \cite[Lemma 3.5]{miller08}. The calculations surrounding this lemma follow \cite[\S 3]{gao} closely. 
At some point in \cite{gao} or \cite{miller08}, the term corresponding to $\sum_p \chi_d(p) \log p /\! \sqrt{p}$ (i.e. the $k = 1$ term of the left hand side of \cref{dirichlet_murmuration_GRH}) is relegated to an error term. This is the term which yields murmurations, explaining one reason why \cite{miller08} is just short of allowing one to prove the existence of murmurations. Indeed, \cite[Thm.\ 1.2]{LOP} and its proof resemble \cite[\S 3]{gao}.

The other restriction leading to the necessity of assuming the ratios conjecture to demonstrate murmurations in the present setting comes from \cite[\S 2]{miller08}'s analysis of the first term on the right hand side of \cref{dirichlet_murmuration_GRH}. As \eqref{kronecker_GRH_main_term} emphasizes, this is the term in which the murmurations manifest. Thus, to drop the need to assume the ratios conjecture here, one would have to show approximate equality between the $k = 1$ term on the left hand side of \cref{dirichlet_murmuration_GRH} (likely following the analysis of \cite[\S 3]{gao} or \cite[\S 3.2.2]{miller08}) and the first term on the right hand side. Perhaps \cite[\S 5]{LOP} would be a good starting point for this endeavour, and it would also be interesting to determine how \eqref{kronecker_GRH_main_term} compares to the formula obtained in \cite[Thm.\ 1.2]{LOP}.

In \cref{sec:GRH,sec:elliptic_curve_quadratic_twist} we present murmurations for holomorphic modular forms in the level aspect and quadratic twists of fixed elliptic curves in a style similar to this section. A ratio conjecture needs to be assumed in \cref{sec:elliptic_curve_quadratic_twist}, but not \cref{sec:GRH}. Once again, the crux of the matter comes down to whether or not the support of $\hat{g}$ can contain $1$ while keeping the error term small. Moreover, several other papers, such as \cite{miller09, gjmmnpp, miller_peckner, fiorilli_miller}, seem like they'd also be amenable to the sort of approach we present here for exhibiting murmurations.


\subsection{Quadratic characters --- GRH}
\label{sec:cech}

In this section we prove \cref{cech_murmuration} using the work of \v Cech \cite{cech}. The main result we will adapt is the following.

\begin{theorem}[{\v Cech \cite[Thm.\ 1.4]{cech}}]\label{cech}
  Assume GRH. Let $f$ be a smooth, fast-decaying (as in \cite{cech}) weight function with Mellin transform $\cM f$, and let $\chi_n$ denote the Jacobi symbol $(\frac{\cdot}{n})$. For $\eps < \Re(r) < \tfrac{1}{4}$ and $N \to \infty$,
  \begin{align*}
    &\sum_{\substack{n \,>\, 0 \\ n \text{\emph{ odd, squarefree}}}} \frac{1}{N}f\!\left(\frac{n}{N}\right) \frac{L'(\half + r, \chi_n)}{L(\half + r, \chi_n)} 
    = \frac{2 \cM f(1)}{3\zeta(2)}\left(\frac{\zeta'(1 + 2r)}{\zeta(1 + 2r)} + \sum_{p > 2} \frac{\log p}{(p+1)p^{1+2r} - 1}\right)
    \\
    &\hspace{3.5cm}- \cM f(1-r) \left(\frac{\pi}{N}\right)^r \left(\frac{\Gamma\!\left(\frac{\half-r}{2}\right)}{\Gamma\!\left(\frac{\half + r}{2}\right)} + \frac{\Gamma\!\left(\frac{\frac{3}{2}-r}{2}\right)}{\Gamma\!\left(\frac{\frac{3}{2}+r}{2}\right)}\right)\frac{\zeta(1-2r)}{4\zeta^{(2)}(2-2r)} + O\!\left(|r|^\eps N^{-2\Re(r) + \eps}\right).
  \end{align*}
\end{theorem}

In contrast with \cref{dirichlet_murmuration_GRH} which features weights on the range of primes being considered, \cref{cech} and hence \cref{cech_murmuration} features weights on the characters being considered. These sorts of weights can be combined; see e.g.\ \cite[Cor.\ 1.5]{cech}.


\begin{proof}[Proof of \cref{cech_murmuration}]
   This proof is similar to the proof of \cref{dirichlet_murmuration}. When $n$ is odd and squarefree $\chi_n$ is a primitive character modulo $n$ \cite[\S 2]{cech}. Averaging \eqref{explicit_formula_dirichlet_1} yields
\begin{align*}
  & \sum_{\substack{n \,>\, 0 \\ n \text{ odd, squarefree}}} \frac{1}{N} f\!\left(\frac{n}{N}\right) \sum_{p^k < \x } \chi_d(p^k)\log p
  \\
  &\hspace{2cm}= -\frac{1}{2\pi i}\int_{c - iT}^{c + iT} \sum_{\substack{n \,>\, 0 \\ n \text{ odd, squarefree}}} \frac{1}{N} f\!\left(\frac{n}{N}\right) \frac{L'(s,\chi_n)}{L(s,\chi_n)} \x^s\,\frac{ds}{s} + O\!\left(\x^{1+\eps} T^{-1+\eps} N^\eps\right).
\end{align*}

Using \cref{cech},
\begin{align*}
  -\frac{1}{2\pi i}\int_{c - iT}^{c + iT} &\sum_{\substack{n \,>\, 0 \\ n \text{ odd, squarefree}}}  \frac{1}{N} f\!\left(\frac{n}{N}\right) \frac{L'(s,\chi_n)}{L(s,\chi_n)} \x^s\,\frac{ds}{s}\\
  ={}&-\frac{1}{2\pi i}\int_{c-iT}^{c+iT} \frac{2 \cM f(1)}{3\zeta(2)}\frac{\zeta'(2s)}{\zeta(2s)} \x^s\,\frac{ds}{s}\\
  &-\frac{1}{2\pi i}\int_{c-iT}^{c+iT} \frac{2 \cM f(1)}{3\zeta(2)}\sum_p \frac{\log p}{(p+1)(p^{2s}-1)} \x^s\,\frac{ds}{s}\\
  &+\frac{\x^\half}{2\pi i}\int_{c-iT}^{c+iT} \cM f(\tfrac{3}{2}-s) \left(\frac{\Gamma\!\left(\frac{1-s}{2}\right)}{\Gamma\!\left(\frac{s}{2}\right)} + \frac{\Gamma\!\left(\frac{2-s}{2}\right)}{\Gamma\!\left(\frac{1+s}{2}\right)}\right)\frac{\zeta(2-2s)}{4\zeta^{(2)}(3-2s)} \left(\frac{\pi \x }{N}\right)^{s - \half} \,\frac{ds}{s}\\
  &-\frac{1}{2\pi i}\int_{c-iT}^{c+iT} O\!\left(|s-\thalf|^\eps N^{1 - 2c + \eps}\right) \x^s\,\frac{ds}{s}.
\end{align*}

Recall \eqref{zeta_explicit_formula} and \eqref{Aderiv_bound}, and note that
\begin{align*}
  &-\frac{1}{2\pi i}\int_{c-iT}^{c+iT} O\!\left(|s-\thalf|^\eps N^{1 - 2c + \eps}\right) \x^s\,\frac{ds}{s} \ll \x^c T^\eps N^{1 - 2c + \eps}.
\end{align*}

Collecting terms then yields \cref{cech_murmuration}.
\end{proof}


\subsection{Holomorphic modular forms in the level aspect --- GRH}
\label{sec:GRH}

In \cite{miller_montague} the authors prove, subject only to GRH, that certain key predictions of ratios conjectures for newspaces of holomorphic modular forms are in agreement with the corresponding arithmetic quantities. These results involve smooth test functions, with error terms that depend on these test functions, as well as some mild weights. By transporting these results into the framework presented here, we can prove that GRH implies the existence of murmurations in this case.

Let $g$ be an even Schwartz function with the property that its Fourier transform
$$\hat{g}(\xi) \coloneqq \frac{1}{2\pi}\int_\R g(x)e^{-2\pi i\xi x}\, dx$$
is supported on $(-\sigma, \sigma)$. 
Let $H_k^+(N)$ and $H_k^-(N)$ denote bases of normalized Hecke eigenforms of weight $k$, level $N$, and root number $\pm 1$ matching the superscript. Let $\lambda_f(p)$ be such that $p^{\frac{k-1}{2}}\lambda_f(p)$ is the $p^{\text{th}}$ Hecke eigenvalue of $f$. 


\begin{proof}[Proof of \cref{newspace_murmuration_GRH}]
  Let $V_1$ be as defined in \cite[(3.15)]{miller_montague}, with $R = N$:
  \begin{align*}
    V_1 \coloneqq \sum_{a\in\{0,1\}} \sum_{f \in H_k^\pm(N)} \frac{\Gamma(k-1) \lambda_f(N^a)}{(4\pi)^{k-1}\langle f,f \rangle} \frac{(i^k\mu(N)N^\half)^a}{\log N} \sum_{p\neq N} \hat{g}\!\left(\frac{\log p}{\log N}\right)\frac{\lambda_f(p) \log p}{\sqrt{p}}
    .
  \end{align*}
  By \cite[(3.22)]{miller_montague}, the $a = 0$ term is $\ll N^{\frac{\sigma}{2} - 1 + \eps}$. By \cite[Prop.\ 14.16, Thm.\ 14.17]{IK} or as discussed in \cite[Appendix C]{miller_montague}, the sign $\rootnum_f$ of the functional equation $\Lambda(s,f) = \rootnum_f \Lambda(1-s, \bar f)$ is $i^k\mu(N)N^\half\lambda_f(N)$. Thus,
  \begin{align}
    \label{eq:V1}
    V_1 = \pm \frac{1}{\log N} \sum_{f \in H_k^\pm(N)} \frac{\Gamma(k-1)}{(4\pi)^{k-1}\langle f,f \rangle} \sum_{p\neq N} \hat{g}\!\left(\frac{\log p}{\log N}\right)\frac{\lambda_f(p) \log p}{\sqrt{p}}
    + O(N^{\frac{\sigma}{2} - 1 + \eps})
    .
  \end{align}

  By \cite[Lemma 3.5]{miller_montague},
  \begin{align*}
    V_1
    &= 2\lim_{\delta \to 0^+} \int_{-\infty}^\infty
    g(t \log N)
    \frac{\Gamma\!\left(\frac{k}{2} - 2\pi it\right)}{\Gamma\!\left(\frac{k}{2} + 2\pi it\right)}
    \prod_p\left(1 + \frac{1}{(p-1)p^{4\pi it + \eps}}\right)
    \left(\frac{2\pi}{\sqrt{N}}\right)^{\!4\pi it}
    \,dt
    + O(N^{\frac{\sigma}{2} - 1 + \eps})
    \\
    &=
    2\lim_{\delta \to 0^+} \int_{-\infty}^\infty
    \int_{-\infty}^\infty
    \hat{g}(\xi) N^{2\pi i \xi t}\,d\xi\,
    \frac{\Gamma\!\left(\frac{k}{2} - 2\pi it\right)}{\Gamma\!\left(\frac{k}{2} + 2\pi it\right)}
    \prod_p\left(1 + \frac{1}{(p-1)p^{4\pi it + \eps}}\right)
    \left(\frac{2\pi}{\sqrt{N}}\right)^{\!4\pi it}
    \,dt
    + O(N^{\frac{\sigma}{2} - 1 + \eps})
    \\
    &=
    2\lim_{\delta \to 0^+}
    \int_{-\infty}^\infty
    \hat{g}(\xi)
    \int_{-\infty}^\infty
    \frac{\Gamma\!\left(\frac{k}{2} - 2\pi it\right)}{\Gamma\!\left(\frac{k}{2} + 2\pi it\right)}
    \prod_p\left(1 + \frac{1}{(p-1)p^{4\pi it + \eps}}\right)
    \left(\frac{4\pi^2 N^\xi}{N}\right)^{\!2\pi it}
    \,dt
    d\xi\,
    + O(N^{\frac{\sigma}{2} - 1 + \eps})
    .
  \end{align*}
  Substituting $N^\xi \eqqcolon Ny \Leftrightarrow \xi = 1 + \frac{\log y}{\log N}$ and combining with \eqref{eq:V1} completes the proof.
\end{proof}

As stated in \cite[(1.3) and (2.6)]{miller_montague},
\begin{align*}
  \# H_k^\pm(N) = \frac{k-1}{24}N + O((kN)^{\frac{5}{6}}) \quad\quad\text{and}\quad\quad N^{-1-\eps} \ll \frac{\Gamma(k-1)}{(4\pi)^{k-1} \langle f, f \rangle} \ll N^{-1+\eps}.
\end{align*}
Moreover, as a consequence of the Petersson trace formula \cite[(A.8)]{miller_montague},
$$\sum_{f \in H_k^\pm(N)} \frac{\Gamma(k-1)}{(4\pi)^{k-1}\langle f,f \rangle} \sim \half.$$
One is therefore able to choose $\hat{g}(\xi)$ to be sharply peaked around $\xi = 1$ while keeping the error term small.


\subsection{Quadratic twists of an elliptic curve --- Ratios conjecture}
\label{sec:elliptic_curve_quadratic_twist}

In \cite{hmm}, Huynh, Miller, and Morrison investigate agreement between the ratios conjecture and the corresponding arithmetic quantities in the case of quadratic twists of a fixed elliptic curve.

Let $N_E$ be an odd prime and let $E$ be an elliptic curve of conductor $N_E$ with even functional equation. Define
$$\cF_E(D) \coloneqq \left\{ 0 < d < D \,:\, \text{$d$ a fundamental discriminant}, \chi_d(-N_E)\rootnum_E = 1 \right\}.$$
Let $\alpha_p$ be as in \cref{sec:elliptic_curves}, and let $A_E$ be as in \cite[(3.2)]{hmm}.

By collecting results from \cite{hmm}, one obtains the following.
\begin{theorem}[{based on Huynh--Miller--Morrison \cite{hmm}}]\label{elliptic_curve_murmuration_GRH}
  Assume GRH. With the notation above,
  \begin{align*}
    &\frac{1}{\#\cF_E(D)}\sum_{d \in \cF_E(D)} \sum_{p \neq N_E} \sum_{\substack{k \,>\, 0 \\ \text{$k$ \emph{odd}} }} \hat{g}\!\left(\frac{\log p^k}{\log\frac{N_E D^2}{4\pi^2}}\right) \chi_d(p) \frac{(\alpha_p^k + \overline{\alpha}_p^k) \log p}{\sqrt{p^k}}\\
    &\hspace{1cm}
    = \frac{g(0)}{2}\log\!\left(\frac{\sqrt{N_E}D}{2\pi}\right)
    + \int_{-\infty}^{\infty} \frac{\Gamma(1 - \pi i t)}{\Gamma(1 + \pi i t)} \frac{\zeta(1 + 2\pi i t) L_E(\mathrm{sym}^2, 1 - 2\pi i t)}{L_E(\mathrm{sym}^2, 1)} A_E(-\pi i t, \pi i t)
    \\
    &\hspace{4.5cm}
    \cdot\frac{1}{\#\cF_E(D)}\sum_{d \in \cF_E(D)} \int_0^\infty \hat{g}\!\left(1 + \frac{\log y}{\log\frac{\sqrt{N_E} D}{2\pi}}\right) \left(\frac{Dy}{d}\right)^{2\pi i t} \,\frac{dy}{y}\,dt
    + O\!\left(D^{-\frac{1-\sigma}{2}}\right).
  \end{align*}
\end{theorem}

\begin{proof}[Proof of \cref{elliptic_curve_murmuration_GRH}]
  By assuming the ratios conjecture, we may equate the right hand sides of theorem 1.2 and lemma 2.1 of \cite{hmm}. The last term in lemma 2.1 is, over the course of section 2, broken up into five parts: $S_{\text{even},1,1}$, $S_{\text{even},1,2}$, $S_{\text{even},2}$, $S_{\text{odd}}(p\mid M)$, and $S_{\text{odd}}(p \nmid M)$. Expressions for the first four of these are given in lemma 2.2, lemma 2.3, lemma 2.4, and (2.34)+(2.37) respectively. Note two typos: the last term in (2.1) and the error term in (2.26) are each missing a factor of $(\x^*)^{-1}$. Comparing term by term with theorem 1.2 yields \cref{elliptic_curve_murmuration_GRH} after some elementary manipulations.
\end{proof}

One can omit the $g(0)/2$ term in \cref{elliptic_curve_murmuration_GRH} by modifying the integral to include a semicircle counterclockwise around $t = 0$ \cite[(2.20)--(2.21)]{hmm}. See also the presentation of \cref{newspace_murmuration_GRH}.

As was the case in \cref{sec:kronecker_2}, the error term of \cref{elliptic_curve_murmuration_GRH} (which is governed by a key bound used in \cite[Lemma 2.5]{hmm}) prevents one from using the theorem to investigate murmurations in the range $p/N \approx 1$. One can nonetheless use \cite[Lemma B.1]{hmm} to rewrite the main term of \cref{elliptic_curve_murmuration_GRH} in the style of \eqref{kronecker_GRH_main_term} and thereby obtain a formula for murmurations of quadratic twists of elliptic curves subject to the associated ratios conjecture.

\subsection{$\zeta'/\zeta$ --- Ratios conjecture}
\label{sec:zeta_logderiv_correlation}

In the previous sections our approach has been to take inverse Mellin transforms of ratios conjectures. The right hand sides of these ratios conjectures have in them a power of the analytic conductor, and it is with respect to this power that we take the inverse Mellin transform to obtain $p/N$-invariance; see \cref{rem:dirichet_functional_equation_factor}.

Ratios conjectures have also been used in the study of correlations of Dirichlet series on vertical lines. In this section, we apply our method to correlations of $\zeta'/\zeta$, using the ratios conjecture \cite[Conj.\ 2.1]{conrey_snaith}. \Cref{sec:zeta_correlation} is closely related, pertaining to correlations of $\zeta$ and based on an unconditional result of \cite{bettin}.

\begin{theorem}[Conrey--Snaith {\cite[Thm.\ 2.5]{conrey_snaith}}]
  \label{zeta_logderiv_correlation}
  Assume the ratios conjecture \cite[Conj.\ 2.1]{conrey_snaith}. For $T \in \R_{>0}$ and $\alpha, \beta \in \C$ satisfying $\frac{1}{\log T} \ll \Re(\alpha), \Re(\beta) < \frac{1}{4}$,
  \begin{align*}
    \int_0^T \frac{\zeta'}{\zeta}&\!\left(\half + it + \alpha\right) \frac{\zeta'}{\zeta}\!\left(\half - it + \beta\right) dt\\
    &= \zeta(1 + \alpha + \beta) \zeta(1 - \alpha - \beta)
    \prod_p \left(1 - \frac{1}{p^{1 + \alpha + \beta}}\right)\left(1 - \frac{2}{p} + \frac{1}{p^{1 + \alpha + \beta}}\right)\left(1 - \frac{1}{p}\right)^{-2}
    \int_0^T \left(\frac{2\pi}{t}\right)^{\alpha + \beta}dt\\
    &\quad+ T\left(\frac{\zeta'}{\zeta}\right)'\!(1 + \alpha + \beta)
    - T\sum_p \left(\frac{\log p}{p^{1 + \alpha + \beta} - 1}\right)^2
    + O(T^{\half + \eps}).
  \end{align*}
\end{theorem}

A version of the result above for correlations of logarithmic derivatives of holomorphic modular form $L$-functions is presented in \cite[Prop.\ 5.2]{bbjmnsy}.

Set $z \coloneqq \alpha + \beta$. Choose $T, \Delta T, H, c > 0$. Taking a partial inverse Mellin transform of \cref{zeta_logderiv_correlation} with respect to $z$ yields \cref{thm:zeta_logderiv_correlation_murmurations}.

In \cref{sec:kronecker_ratiosconjecture,sec:elliptic_curves,sec:GRH,sec:elliptic_curve_quadratic_twist} we considered families $\cF$ of arithmetic objects whose $L$-functions had conductors of similar size. Our strategy for demonstrating the existence of murmurations was to evaluate expressions of the form
\begin{align}
  \label{eq:sample_inverse_mellin}
  \frac{1}{2\pi i} \int_{c - iT}^{c + iT} \frac{1}{\#\cF} \sum_{f \in \cF} \frac{L'(s,f)}{L(s,f)}\x^s\,\frac{ds}{s}
\end{align}
in two different ways:
\begin{enumerate}[label=(\roman*)]
\item
  by using an explicit formula (which approximates \eqref{eq:sample_inverse_mellin} as an average over $\cF$ of Dirichlet coefficients), and
\item
  by using a ratios conjecture (which approximates \eqref{eq:sample_inverse_mellin} by some function which exhibits $\x /N$ scale-invariance, where $N$ is the size of the conductors in $\cF$).
\end{enumerate}

\Cref{thm:zeta_logderiv_correlation_murmurations}'s left hand side plays the role of \eqref{eq:sample_inverse_mellin} in the setting of this section. The theorem's right hand side follows from the ratios conjecture \cite[Conj.\ 2.1, Thm.\ 2.5]{conrey_snaith}, and is analogous to the point (ii) listed above.

One might ask if there is an interpretation of the left hand side of \cref{thm:zeta_logderiv_correlation_murmurations} parallel to point (i) above, i.e.\ analogous to the average over $\cF$ of Dirichlet coefficients one obtains from explicit formulas. Largely mirroring \cite[\S 4]{conrey_snaith},
in \cref{prop:zeta_logderiv_series,cor:zeta_logderiv_series} 
we express the left hand side of \cref{thm:zeta_logderiv_correlation_murmurations} by expressions reminiscent of pair correlation for $\zeta$. We first introduce some notation

Let $\x , c_1, c_2, c_3, c_4, T_1, T_2, T_3, T_4$ be real numbers satisfying
\begin{itemize}
\item
  $\x  > 1$,
\item
  $2 \geqslant c_1 > c_2 \geqslant -1$ and $2 \geqslant c_3 > c_4 \geqslant -1$,
\item
  $-T_1 < T_2$ and $-T_3 < T_4$.
\end{itemize}
Define
\begin{align*}
  \Omega_u &\coloneqq \{u \in \C \,:\, c_2 < \Re(u) < c_1,\,-T_1 < \Im(u) < T_2\},
  \\
  \Omega_v &\coloneqq \{v \in \C \,:\, c_4 < \Re(v) < c_3,\,-T_3 < \Im(v) < T_4\},
  \shortintertext{and}
  \delta &\coloneqq \min\!\left\{ |s - \rho|\,:\, s \in \dee\Omega_u \cup \dee\Omega_v,\,\zeta(\rho) = 0 \right\}.
\end{align*}
Impose further that $c_1,\dots, c_4, T_1, \dots, T_4$ are chosen such that $\delta > 0$.
(For any $T \geqslant 2$ there exists $T_* \in [T, T+1]$ such that $|\sigma + iT_* - \rho| \gg \frac{1}{\log T}$ for all zeros $\rho$ of $\zeta$, uniformly for $-1 \leqslant \sigma \leqslant 2$ \cite[Thm.\ 10.13]{MV}.)

Let $f: \Omega_u \times \Omega_v \to \C$ be holomorphic in each variable.

\begin{proposition}
  \label{prop:zeta_logderiv_series}
  With the notation above,
  \begin{align*}
    (2\pi i)^2 \sum_{\substack{\rho_1 \in \Omega_u,\, \rho_2 \in \Omega_v \\ \zeta(\rho_1) = \zeta(\rho_2) = 0}} f(\rho_1, \rho_2) \x^{\rho_1 + \rho_2}
    = {}&\int_{c_1 - iT_1}^{c_1 + iT_2} \int_{c_3 - iT_3}^{c_3 + iT_4} \frac{\zeta'}{\zeta}(u) \frac{\zeta'}{\zeta}(v) f(u,v) \x^{u+v} \,dv\,du
    \\&
    + O\!\left(
    \delta^{-2}
    \int\limits_{\substack{u \in \dee\Omega_u \\ \Re(u) \neq c_1}} \int\limits_{\substack{v \in \dee\Omega_v \\ \Re(v) \neq c_3}} \big|f(u,v) \x^{u+v}\big|\,dv\,du
    \right).
  \end{align*}
\end{proposition}
\begin{proof}
  Follows from \cite[Lemmas 12.1 and 12.2]{MV} and Cauchy's residue theorem.
\end{proof}

\begin{corollary}
  \label{cor:zeta_logderiv_series}
  Let $\x , T, \Delta T, H, c \in \R$ and $\alpha \in \C$ be such that
  \begin{itemize}
  \item
    $\x , T, \Delta T, H > 2$,
  \item
    $H - (T + \Delta T) \gg 1$,
  \item
    $\frac{1}{\log T} \ll \Re(\alpha) < \frac{1}{4}$,
  \item
    $2\Re(\alpha) < c < \frac{1}{4} + \Re(\alpha)$,
  \item
    the distance from any of $H$, $-H$, $T$, or $T+\Delta T$ to the imaginary part of any zero of $\zeta$ is $\gg (\log H)^{-1}$,
  \item
    there are no zeros of $\zeta$ with real part $\half - \Re(\alpha) + c$ and imaginary part between $-H$ and $H$, and
  \item
    there are no zeros of $\zeta$ with real part $\half + \Re(\alpha)$ and imaginary part between $T$ and $T+\Delta T$.
  \end{itemize}
  Then
  \begin{align*}
    \frac{1}{2\pi i} \int_{c - iH}^{c + iH} &\int_T^{T+\Delta T} \frac{\zeta'}{\zeta}\!\left(\half + it + \alpha\right) \frac{\zeta'}{\zeta}\!\left(\half - it - \alpha + z\right) dt \,\x^{z} \,\frac{dz}{z+1}
    \\
    ={}&
    \sum_{\substack{\rho_1\,:\, \zeta(\rho_1) = 0 \\ T < \Im(\rho_1) < T + \Delta T \\ -1 < \Re(\rho_1) < \half + \Re(\alpha)}}
    \:
    \sum_{\substack{\rho_2\,:\, \zeta(\rho_2) = 0 \\ -H < \Im(\rho_2) < H \\ -1 < \Re(\rho_2) < \half - \Re(\alpha) + c}}
    2\pi \frac{\x^{\rho_1 + \rho_2 - 1}}{\rho_1 + \rho_2}
    \quad
    +
    \quad
    O\!\left(
    \frac{\x^{c}\Delta T}{H - (T + \Delta T)} + \x^{c}\log H
    \right).
  \end{align*}
\end{corollary}
The condition that the distance from any of $H$, $-H$, $T$, or $T+\Delta T$ to the nearest imaginary part of a zero of $\zeta$ be $\gg (\log H)^{-1}$ can always be satisfied by increasing each of $H$, $T$, and $\Delta T$ by at most $1$ \cite[Thm.\ 10.13]{MV}. One could also consider a variant of \cref{prop:zeta_logderiv_series} which does not impose this condition and instead has in its error term a dependence on proximity to $\zeta$ zero ordinates.

\begin{proof}[Proof of \cref{cor:zeta_logderiv_series}]
  In \cref{prop:zeta_logderiv_series} take
  \begin{align*}
    f(u,v) &= \frac{1}{u+v}
    \\
    -T_1, T_2, -T_3, T_4 &= T, T+\Delta T, -H, H
    \\
    c_1, c_2, c_3, c_4 &= \half + \Re(\alpha),\, -0.99,\, \half - \Re(\alpha) + c,\, -1. \qedhere
  \end{align*}
\end{proof}

Combining \cref{cor:zeta_logderiv_series} and \cref{thm:zeta_logderiv_correlation_murmurations} yields the following.

\begin{proposition}
  \label{prop:zeta_logderiv_correlation_eval}
  Assume the ratios conjecture \cite[Conj.\ 2.1]{conrey_snaith}. With the same notation as \cref{cor:zeta_logderiv_series},
\begin{align*}
  &
  \sum_{\substack{\gamma_1\,:\, \zeta(\half + i\gamma_1) = 0 \\ T < \gamma_1 < T + \Delta T}}
  \:\:
  \sum_{\substack{\gamma_2\,:\, \zeta(\half + i\gamma_2) = 0 \\ |\gamma_2| < H}}
  \:\:
  2\pi \frac{\x^{i(\gamma_1 - \gamma_2)}}{1 + i(\gamma_1 - \gamma_2)}
  \\
  &\hspace{1cm}
  =
  \frac{1}{2\pi i} \int_{c - iH}^{c + iH}\zeta(1 + z) \zeta(1 - z)
  \prod_p \left(1 - \frac{1}{p^{1 + z}}\right)\left(1 - \frac{2}{p} + \frac{1}{p^{1 + z}}\right)\left(1 - \frac{1}{p}\right)^{-2}
  \int_T^{T+\Delta T} \left(\frac{2\pi \x }{t}\right)^z dt \,\frac{dz}{z+1}
  \\
  &\hspace{1cm}
  \quad
  + O\!\left(
  \x^c H^\eps (T^{\half + \eps} + \Delta T)
  \right)
  .
\end{align*}
\end{proposition}

\subsection{$\zeta$ --- Unconditional}
\label{sec:zeta_correlation}

We repeat the approach presented in \cref{sec:zeta_logderiv_correlation} for the function $\zeta$ in place of $\zeta'/\zeta$.
Asymptotics for the shifted second moment of $\zeta$ on vertical lines was first established by Ingham \cite{ingham}, for bounded shifts. A version allowing for unbounded shifts, which is necessary for our application, is proved by Bettin in \cite{bettin}.

\begin{theorem}[{Bettin \cite{bettin}}]
  \label{thm:bettin}
  For $T > 2$ and $\alpha, \beta \in \C$ with $\Re(\alpha), \Re(\beta) \ll \frac{1}{\log T}$,
  \begin{align*}
    \int_0^T \zeta\!\left(\thalf + it + \alpha\right)& \zeta\!\left(\thalf - it - \beta\right)\,dt\\
    = {}&\int_0^T \zeta(1+\alpha-\beta) + \zeta(1-\alpha+\beta) \pi^{\alpha - \beta} \frac{\Gamma\!\left(\frac{\half - \alpha - it}{2}\right) \Gamma\!\left(\frac{\half + \beta + it}{2}\right)}{\Gamma\!\left(\frac{\half + \alpha + it}{2}\right) \Gamma\!\left(\frac{\half - \beta - it}{2}\right)}\,dt\\
    &+ O\!\left(\left(T^\half + |\Im(\alpha)|^\half + |\Im(\beta)|^\half\right)(\log T)^2\right).
  \end{align*}
\end{theorem}
Bettin comments that versions of \cref{thm:bettin} which loosen the requirement $\Re(\alpha), \Re(\beta) \ll \frac{1}{\log T}$ could likely be produced using techniques similar to those in \cite{bettin}.

Moreover, \cite[Lemma 2.1]{bettin} relates the quantities in \cref{thm:bettin} to divisor sums. Taking inverse Mellin transforms of these expressions yields murmuration-like scale invariance, which ultimately originates from the functional equation for $\zeta$. We present this in \cref{thm:zeta_correlation_murmurations}, but let's first establish some notation.

Let $\mu(\sigma)$ be such that $\zeta(\sigma + it) \ll |t|^{\mu(\sigma)} + 1$. The Phragm\'en--Lindel\"of principle \cite[Thm.\ 8.2.1]{goldfeld:book} implies that $\mu(\sigma)$ can be taken to be $\thalf(1 - \sigma) + \eps$ for $0 < \sigma < 1$, while the Riemann Hypothesis for $\zeta$ allows one to take $\mu(\sigma) = K/\log\log|t|$ for some positive constant $K$, $\thalf \leqslant \sigma \leqslant 1$, and $|t| > 4$ \cite[Thm.\ 13.18]{MV}.

Let
\begin{align}
  \label{eq:onedef}
  \one{\text{condition}} \coloneqq \begin{cases} 1 & \text{condition is true} \\ 0 & \text{condition is false.} \end{cases}
\end{align}
For example, $\one{c > 0}$ is $1$ if $c > 0$ and $0$ otherwise.

Let $\{\x \}$ denote the distance from $\x $ to the nearest integer. Let $\gamma = 0.577\!\:\!...$ 
denote the Euler--Mascheroni constant.

\begin{theorem}
  \label{thm:zeta_correlation_murmurations}
  For $T, \Delta T, \x , H, c \in \R$ and $\alpha \in \C$ satisfying
  $T \gg \Delta T > 2H > 0$, $T \geqslant 2$, $\x  > 1$, $\x  \not\in \Z$, $H + |\Im(\alpha)| < \thalf T$, $c, \Re(\alpha) \ll \frac{1}{\log T}$, $c \neq 0$,
  \begin{align}
    \nonumber
    \frac{1}{2\pi i} \int_{c - iH}^{c + iH} \int_{T - \Im(\alpha)}^{T+\Delta T - \Im(\alpha)}
    &\zeta\!\left(\thalf + it + \alpha\right)\zeta\!\left(\thalf - it - \alpha + z\right)\,dt\,\x^z\frac{dz}{z}
    \\
    \label{eq:zeta_correlation_eq1}
    ={}&
    \sum_{\frac{T + H}{2\pi} < n < \frac{T + \Delta T - H}{2\pi}} 2\pi \sum_{d\mid n} \one{c > 0, d < \x} - \one{c < 0, d > \x}
    \\&
    \nonumber
    + O\!\left(
    \x^c T^\half (\log T)^2
    \log\tfrac{H}{c}
    + \x^c \frac{\log T}{\log\log T}
    (\log\tfrac{H}{c} + T^{-c}) (H + H^{-1} \Delta T \{\x \}^{-1} )
    \right)
    \\
    \label{eq:zeta_correlation_eq2}
    ={}&
    \frac{1}{2\pi i} \int_{c-iH}^{c+iH} \zeta(1-z) \int_{T + H}^{T + \Delta T - H} \left(\frac{2\pi \x }{t}\right)^z dt\,\frac{dz}{z}
    \\&
    \nonumber
    -\one{c < 0}(\log \x  + \gamma)\Delta T  + \Delta T \sum_{d < \x } \frac{1}{d}
    \\&
    \nonumber
    + O\Big(
    \x^c T^\half (\log T)^2
    \log\tfrac{H}{c}
    + \x^c T^{-c} H^{1 + \mu(1-c)}
    \log\tfrac{H}{c}
    + (\x^c + 1)H^{-1 + \mu(1 + c)}\Delta T
    \\&
    \nonumber
    \hspace{4.5cm}
    + \Delta T \min\!\left\{\x^{-1}, H^{-1}\{\x \}^{-1} \right\}
    + 4^{\frac{1}{\log \x }}H^{-1} \Delta T \log \x 
    \Big).
  \end{align}
\end{theorem}

Both error terms are $\ll (T^\half + H + \Delta T / H)(\x TH)^\eps$ when $-\log c \ll T^{\eps}$ and $\x $ is bounded away from integers.

We present the \hyperlink{proof:bettin1}{proof of \eqref{eq:zeta_correlation_eq1}} and the \hyperlink{proof:bettin2}{proof of \eqref{eq:zeta_correlation_eq2}} separately.

\begin{proof}[\hypertarget{proof:bettin1}{Proof of \eqref{eq:zeta_correlation_eq1}}]
  We introduce some notation of \cite{bettin} so that we may state \eqref{eq:bettin1} and \eqref{eq:bettin2} below.
  
  Let $a,b \in \C$ be such that $\Re(a), \Re(b) \ll \frac{1}{\log T}$ and $\Im(a), \Im(b) \ll T^2$. Let $u,v \in \R$ be such that $T^u = \max\{T, |\Im(a)|\}$ and $T^v = \max\{T, |\Im(b)|\}$.

  Let $S \subset \Z^2$ denote pairs $(m,d)$ of positive integers such that
  \begin{align*}
    \begin{cases}
      \frac{\Im(b)}{2\pi} < md \leqslant \frac{T + \Im(b)}{2\pi} & \text{if $\Im(b) > 0$} \\
      md \leqslant \frac{T + \Im(b)}{2\pi} \;\;\;\text{or}\;\;\; md \leqslant \frac{-\Im(b)}{2\pi} & \text{if $-T \leqslant \Im(b) \leqslant 0$} \\
      -\frac{T+\Im(b)}{2\pi} \leqslant md < -\frac{\Im(b)}{2\pi} & \text{if $\Im(b) \leqslant -T$.}
    \end{cases}
  \end{align*}
  
  \cite[Proof of thm.\ 1]{bettin} gives
  \begin{align}
    \label{eq:bettin1}
    \int_0^T \zeta\!\left(\thalf + it + a\right) \zeta\!\left(\thalf - it - b\right)\,dt
    = 2\pi \sum_S d^{b-a} + O(T^{\frac{v}{2}} (\log T)^2 + T^{r(u+v)}),
  \end{align}
  where $r$ is such that $\zeta(\thalf + it) \ll |t|^{r+\eps}$. By a result of Hardy and Littlewood (see \cite{huxley_ivic}) we may take $r = \tfrac{1}{6}$.

  \Cref{thm:zeta_correlation_murmurations} will take $a = \alpha$ and $b = \alpha - z$.
  Our assumption $H + |\Im(\alpha)| < \thalf T$ guarantees that $u = v = 1$, and that $S$ is such that \eqref{eq:bettin1} will imply
  \begin{align}
    \label{eq:bettin1_consequence}
    \int_T^{T+\Delta T} \zeta\!\left(\thalf + it + \alpha\right) \zeta\!\left(\thalf - it - \beta\right)\,dt
    = \sum_{\frac{T + \Im(\alpha - z)}{2\pi} < n < \frac{T + \Delta T + \Im(\alpha - z)}{2\pi}} 2\pi \sigma_{-z}(n) + O(T^{\frac{1}{2}} (\log T)^2).
  \end{align}
  Our assumption $\Delta T \ll T$ justifies writing $T$ instead of $T + \Delta T$ in the error term above.

  We now take an inverse Mellin transform of \eqref{eq:bettin1_consequence} in $z$.
  \begin{align}
    \nonumber
    \frac{1}{2\pi i}\int_{c - iH}^{c + iH} \int_T^{T+\Delta T} &\zeta\!\left(\thalf + it + \alpha\right) \zeta\!\left(\thalf - it - \beta\right)\,dt\,\x^z\,\frac{dz}{z}
    \\
    \label{eq:bettin1_inversemellin_mt}
    = {}&\frac{1}{2\pi i}\int_{c - iH}^{c + iH} \sum_{\frac{T + \Im(\alpha - z)}{2\pi} < n < \frac{T + \Delta T + \Im(\alpha - z)}{2\pi}} 2\pi \sum_{d\mid n}\left(\frac{\x }{d}\right)^z\frac{dz}{z}
    \\&\label{eq:bettin1_inversemellin_error}
    + \frac{1}{2\pi i}\int_{c - iH}^{c + iH} O(T^{\frac{1}{2}} (\log T)^2)\,\x^z\,\frac{dz}{z}.
  \end{align}

  We bound \eqref{eq:bettin1_inversemellin_error} using the triangle inequality:
  \begin{align*}
    \frac{1}{2\pi i}\int_{c - iH}^{c + iH} O(T^{\frac{1}{2}} (\log T)^2)\,\x^z\,\frac{dz}{z}
    &
    \ll \x^c T^{\frac{1}{2}} (\log T)^2 \int_{c - iH}^{c + iH} \frac{dz}{|z|}
    \\&
    \ll \x^c T^{\frac{1}{2}} (\log T)^2 \int_{-H}^{H} \frac{dh}{\sqrt{c^2 + h^2}}
    \\&
    \ll \x^c T^{\frac{1}{2}} (\log T)^2 \log\tfrac{H}{c}.
  \end{align*}

  We now consider the main term \eqref{eq:bettin1_inversemellin_mt}. We split the sum over $n$ up as
  \begin{align}
    \nonumber
    \frac{1}{2\pi i}\int_{c - iH}^{c + iH} &\sum_{\frac{T + \Im(\alpha - z)}{2\pi} < n < \frac{T + \Delta T + \Im(\alpha - z)}{2\pi}} \sum_{d\mid n}\left(\frac{\x }{d}\right)^z\frac{dz}{z}
    \\={}&
    \label{eq:mt1_mt}
    \frac{1}{2\pi i}\int_{c - iH}^{c + iH} \sum_{\frac{T + \Im(\alpha) + H}{2\pi} < n < \frac{T + \Delta T + \Im(\alpha) - H}{2\pi}} \sum_{d\mid n}\left(\frac{\x }{d}\right)^z\frac{dz}{z}
    \\&
    \label{eq:mt1_error}
    + \frac{1}{2\pi i}\int_{c - iH}^{c + iH} \sum_{\substack{\frac{T + \Im(\alpha - z)}{2\pi} < n < \frac{T + \Delta T + \Im(\alpha - z)}{2\pi} \\ n \not\in \left(\frac{T + \Im(\alpha) + H}{2\pi}, \frac{T + \Delta T + \Im(\alpha) - H}{2\pi}\right)}} \sum_{d\mid n}\left(\frac{\x }{d}\right)^z\frac{dz}{z}.
  \end{align}
  
  The number of integers $n$ being summed in \eqref{eq:mt1_error} is at most $4H/2\pi + 2$. Let $n$ be such that $T + \Im(\alpha) - H < 2\pi n < T + \Delta T + \Im(\alpha) + H$. Note that $n \ll T$, using our assumptions on the relative sizes of $T$, $\Delta T$, $H$, and $|\Im(\alpha)|$. For any $d \mid n$, we have
  \begin{align*}
    \left(\frac{\x }{d}\right)^z \ll \x^c + \left(\frac{2\pi \x }{T}\right)^c.
  \end{align*}
  We must consider both terms, since $c$ may be either positive or negative. Moreover, Wigert \cite{wigert} proves that $n$ has $\ll \log n / \log\log n$ divisors.

  With the observations of the previous paragraph, we may bound
  \begin{align*}
    \eqref{eq:mt1_error}
    &
    \ll \frac{1}{2\pi i}\int_{c - iH}^{c + iH} H \frac{\log T}{\log\log T} \left(\x^c + \left|\frac{2\pi \x }{T}\right|^c\right) \frac{dz}{|z|}
    \\&
    \ll \x^c H \frac{\log T}{\log\log T} \log\tfrac{H}{c} + H \frac{\log T}{\log\log T} \left|\frac{2\pi \x }{T}\right|^c \log\tfrac{H}{c}.
  \end{align*}

  Next we examine \eqref{eq:mt1_mt}. By complex analysis like that used to produce Perron's formula (see \cite[\S 5]{MV}), we have
  \begin{align*}
    \eqref{eq:mt1_mt}
    &
    =
    \sum_{\frac{T + \Im(\alpha) + H}{2\pi} < n < \frac{T + \Delta T + \Im(\alpha) - H}{2\pi}} \sum_{d\mid n} \frac{1}{2\pi i}\int_{c - iH}^{c + iH} \left(\frac{\x }{d}\right)^z\frac{dz}{z}
    \\&
    =
    \sum_{\frac{T + \Im(\alpha) + H}{2\pi} < n < \frac{T + \Delta T + \Im(\alpha) - H}{2\pi}} \sum_{d\mid n}
    \left[
    \begin{cases}
      \phantom{-}1 & \text{if $c > 0$ and $\x  > d$} \\
      \phantom{-}0 & \text{if $c > 0$ and $\x  < d$} \\
      \phantom{-}0 & \text{if $c < 0$ and $\x  > d$} \\
      -1 & \text{if $c < 0$ and $\x  < d$}
    \end{cases}
    \;\;+\;\;
    O\!\left(
    \x^c d^{-c} \min\!\left\{1, (H\log(\x /d))^{-1} \right\}
    \right)
    \right]\!.
  \end{align*}
  The sum over $n$ contains $\ll \Delta T$ terms, and, for each $n$, the sum over $d$ contains $\ll \log T / \log\log T$ terms by \cite{wigert}. These estimates use our assumptions that $T \gg \Delta T \gg H$ and $T \gg \Im(\alpha)$.
\end{proof}

\begin{proof}[\hypertarget{proof:bettin2}{Proof of \eqref{eq:zeta_correlation_eq2}}]
  Recall the notation introduced at the beginning of the \hyperlink{proof:bettin1}{proof of \eqref{eq:zeta_correlation_eq1}} above.

  \cite[Proof if thm.\ 1]{bettin} gives
  \begin{align}
    \label{eq:bettin2}
    \int_0^T \zeta\!\left(\thalf + it + a\right) \zeta\!\left(\thalf - it - b\right) dt
    ={}&
    \int_0^T \zeta(1 + a - b) + \zeta(1 - a + b)\left(\frac{|t + \Im(b)|}{2\pi}\right)^{b-a} dt
    \\&\nonumber
    +
    O(T^{\frac{u}{2}} (\log T)^2 + T^{\frac{v}{2}} (\log T)^2).
  \end{align}
  Taking $a = \alpha$ and $b = \alpha - z$ as before, \eqref{eq:bettin2} implies, with our assumptions, that
  \begin{align}
    \nonumber
    \int_T^{T + \Delta T} \zeta\!\left(\thalf + it + \alpha\right) &\zeta\!\left(\thalf - it - \alpha + z\right) dt
    \\&
    \nonumber
    = \int_{T + \Im(\alpha) - \Im(z)}^{T + \Delta T + \Im(\alpha) - \Im(z)} \zeta(1 + z) + \zeta(1 - z)\left(\frac{2\pi}{t}\right)^{z} dt
    + O(T^{\frac{1}{2}} (\log T)^2)
    \\&
    \nonumber
    = \Delta T \zeta(1 + z) + \zeta(1 - z) \int_{T + \Im(\alpha) - \Im(z)}^{T + \Delta T + \Im(\alpha) - \Im(z)} \left(\frac{2\pi}{t}\right)^{z} dt
    + O(T^{\frac{1}{2}} (\log T)^2).
  \end{align}
  Taking an inverse Mellin transform in $z$,
  \begin{align}
    \nonumber
    \frac{1}{2\pi i}\int_{c - iH}^{c + iH}
    \int_T^{T + \Delta T} \zeta\!\left(\thalf + it + \alpha\right) &\zeta\!\left(\thalf - it - \alpha + z\right) dt\,\x^z\,\frac{dz}{z}
    \\
    \label{eq:bettin2_t1}
    ={}&
    \frac{1}{2\pi i}\int_{c - iH}^{c + iH} \Delta T \zeta(1 + z)\,\x^z\,\frac{dz}{z}
    \\&
    \label{eq:bettin2_t2}
    + \frac{1}{2\pi i}\int_{c - iH}^{c + iH} \zeta(1 - z) \int_{T + \Im(\alpha) - \Im(z)}^{T + \Delta T + \Im(\alpha) - \Im(z)} \left(\frac{2\pi}{t}\right)^{z} dt\,\x^z\,\frac{dz}{z}
    \\&
    \label{eq:bettin2_t3}
    + \frac{1}{2\pi i}\int_{c - iH}^{c + iH}O(T^{\frac{1}{2}} (\log T)^2)\,\x^z\,\frac{dz}{z}.
  \end{align}

  We analyzed \eqref{eq:bettin2_t3} = \eqref{eq:bettin1_inversemellin_error} above, finding that $\eqref{eq:bettin2_t3} \ll \x^c T^\half (\log T)^2 \log\tfrac{H}{c}$.
  
  We analyze \eqref{eq:bettin2_t1} with Perron's formula \cite[Cor.\ 5.3]{MV}. Shift the contour of integration from $\Re(z) = c$ to $\Re(z) = \frac{1}{\log \x }$. This incurs an error $H^{-1 + \mu(1 + c)}$ from the horizontal segments. If $c < 0$, we also pick up a term coming from the integrand's pole at $z = 0$. Applying \cite[Cor.\ 5.3]{MV} then gives
  \begin{align*}
    \eqref{eq:bettin2_t1}
    ={}& -\one{c < 0}(\log \x  + \gamma)\Delta T  + \Delta T \sum_{d < \x } \frac{1}{d}
    \\&
    \nonumber
    + O\!\left((\x^c + 1)H^{-1 + \mu(1 + c)}\Delta T + 4^{\frac{1}{\log \x }}H^{-1} \Delta T \log \x  +  \Delta T \min\!\left\{\x^{-1}, H^{-1}\{\x \}^{-1} \right\}\right).
  \end{align*}

  We analyze \eqref{eq:bettin2_t2} similarly to \eqref{eq:bettin1_inversemellin_mt} above:
  \begin{align}
    \nonumber
    \frac{1}{2\pi i}\int_{c - iH}^{c + iH} &\zeta(1 - z) \int_{T + \Im(\alpha) - \Im(z)}^{T + \Delta T + \Im(\alpha) - \Im(z)} \left(\frac{2\pi}{t}\right)^{z} dt\,\x^z\,\frac{dz}{z}
    \\
    \nonumber
    ={}&\frac{1}{2\pi i}\int_{c - iH}^{c + iH} \zeta(1 - z) \int_{T + \Im(\alpha) + H}^{T + \Delta T + \Im(\alpha) - H} \left(\frac{2\pi \x }{t}\right)^{z} dt\,\frac{dz}{z}
    \\&
    \label{eq:mt2_error}
    + \frac{1}{2\pi i}\int_{c - iH}^{c + iH} \zeta(1 - z) \underset{t \not\in \big(T + \Im(\alpha) + H,\, T + \Delta T + \Im(\alpha) - H\big)}{\int_{T + \Im(\alpha) - \Im(z)}^{T + \Delta T + \Im(\alpha) - \Im(z)}} \left(\frac{2\pi \x }{t}\right)^{z} dt\,\frac{dz}{z}.
  \end{align}
  The length of the path of integration (in $t$) in \eqref{eq:mt2_error} is bounded by $4H$. The integrand is $\ll \left|\frac{2\pi \x }{T}\right|^c$ (using our assumptions on the relative size of $T$ and the other quantities). Thus, we find that
  \begin{align*}
    &\eqref{eq:mt2_error} \ll H^{1 + \mu(1 - c)} \left|\frac{2\pi \x }{T}\right|^c \log\tfrac{H}{c}. \qedhere
  \end{align*}
\end{proof}

\section{Approximate functional equations}
\label{sec:afe}

The ratios conjecture recipe of \cite{conrey_snaith} leads to scale-invariance characteristic of murmurations specifically because of the recipe's use of the approximate functional equation. The factor $N_f^{\half - s}$ which results from the use of the functional equation is multiplied by $\x^\half \x^{s - \half}$ upon taking an inverse Mellin transform, yielding the $\x /N$-invariance (cf.\ \cref{rem:dirichet_functional_equation_factor}).

We elaborate. The ratios conjecture proper (cf. \cref{ratios_conjecture_dirichlet}) is the following.
\begin{conjecture}[{Conrey--Snaith \cite[Conj.\ 2.6]{conrey_snaith}}]
  \label{conj:ratios_conjecture_quadratic}
  Let $D > 1$ and define
  $$\cF(D) \coloneqq \{d \,:\, 1 < d < D,\,\text{$d$ is a fundamental discriminant}\}.$$
  Let $\alpha, \gamma \in \C$ with $-\tfrac{1}{4} < \Re(\alpha) < \tfrac{1}{4}$, $\frac{1}{\log D} \ll \Re(\gamma) < \tfrac{1}{4}$, and $\Im(\alpha), \Im(\gamma) \ll D^{1-\eps}$. Then
  \begin{align*}
    &\frac{1}{\#\cF(D)} \sum_{d \in \cF(D)} \frac{L(\half + \alpha, \chi_d)}{L(\half + \gamma, \chi_d)}
    \\
    &=
    \frac{\Gamma\!\left(\frac{1}{4} - \frac{\alpha}{2}\right)}{\Gamma\!\left(\frac{1}{4} + \frac{\alpha}{2}\right)}
    \frac{\zeta(1-2\alpha)}{\zeta(1 - \alpha + \gamma)}
    \prod_p \left(1 - \frac{1}{p^{1 - \alpha + \gamma}}\right)^{-1}\left(1 - \frac{1}{(p+1)p^{1 - 2\alpha}} - \frac{1}{(p+1)p^{-\alpha + \gamma}}\right)
    \frac{1}{\#\cF(D)} \sum_{d \in \cF(D)} \left(\frac{\pi}{d}\right)^{\alpha}
    \\
    &+ \frac{\zeta(1 + 2\alpha)}{\zeta(1 + \alpha - \gamma)}
    \prod_p \left(1 - \frac{1}{p^{1 + \alpha + \gamma}}\right)^{-1}\left(1 - \frac{1}{(p+1)p^{1 + 2\alpha}} - \frac{1}{(p+1)p^{\alpha + \gamma}}\right)
    \,+\,
    O(D^{\half + \eps})
    .
  \end{align*}
\end{conjecture}
The factor $\left(\frac{\pi}{d}\right)^{\alpha}$ can be leveraged to produce murmurations by taking an inverse Mellin transform, as we have done repeatedly.

\Cref{conj:ratios_conjecture_quadratic} comes from writing
\begin{align}
  \label{eq:ratios_conjecture_decomp}
  L(s, \chi_d) \cdot \frac{1}{L(z, \chi_d)}
  &=
  \left(
  \sum_{n=1}^\infty \frac{\chi_d(n)}{n^s} V_{}\!\left(\frac{n}{\sqrt{d}}\right)
  + \frac{\halfGamma{1-s}}{\halfGamma{s}}
  \sum_{n=1}^\infty \frac{\chi_d(n)}{n^{1-s}} V_{1-s}\!\left(\frac{n}{\sqrt{d}}\right)
  \left(\frac{\pi}{d}\right)^{s-\half}
  \right)
  \cdot
  \sum_{m=1}^\infty \frac{\mu(m)\chi_d(m)}{m^z}
  ,
\end{align}
where the approximate functional equation (which we recall below) has been applied to $L(s,\chi_d)$. The Dirichlet series are then multiplied together, an average over $d \in \cF(D)$ is taken, and heuristically only contributions from $\chi_d(\square)$ are preserved. See \cite[\S 2.2]{conrey_snaith} for details.

Inspecting \eqref{eq:ratios_conjecture_decomp}, we see that the approximate functional equation is the source of the factor $\left(\frac{\pi}{d}\right)^{s-\half}$ which yields scale invariance.

We are thus led to seek a connection between murmurations and the approximate functional equation which is not confined to the study of ratios of Dirichlet series. Our strategy will be to evaluate expressions of the form
\begin{align}
  \label{eq:avg_mellin}
  \frac{1}{2\pi i}
  \int_{c - i\infty}^{c + i\infty}
  \frac{1}{\#\cF} \sum_{\rep \in \cF} L(s,\rep)\,\x^s f(s)ds 
\end{align}
(i.e.\ the average over a family of $L$-functions of a truncated inverse Mellin transform) in two different ways:
\begin{enumerate}[label=(\roman*)]
\item
  Using Mellin inversion, e.g.\ Perron's formula, to express \eqref{eq:avg_mellin} as an average of Dirichlet coefficients.
\item
  Using the approximate functional equation, to express \eqref{eq:avg_mellin} as an explicit meromorphic function.
\end{enumerate}

\subsubsection*{Mellin inversion}

If $L(s,\rep)$ is an $L$-function with Dirichlet series
\begin{align*}
  L(s,\rep) = \sum_{n=1}^\infty \frac{a_\rep(n)}{n^s}
\end{align*}
absolutely convergent whenever $\sigma \geqslant c_0$, and $f$ is a meromorphic function which decays sufficiently quickly, then
\begin{align}
  \nonumber
  \frac{1}{2\pi i} \int_{c_0 - i\infty}^{c_0 + i\infty} L(s,\rep) x^sf(s)ds
  &=
  \frac{1}{2\pi i} \int_{c_0 - i\infty}^{c_0 + i\infty} \sum_{n=1}^\infty \frac{a_\rep(n)}{n^s} x^sf(s)ds
  \\
  \nonumber
  &= \sum_{n=1}^\infty a_\rep(n) \frac{1}{2\pi i} \int_{c_0 - i\infty}^{c_0 + i\infty} \left(\frac{n}{x}\right)^{-s}f(s)ds
  ,
  \shortintertext{which we recognize as}
  \label{eq:mellin_inversion}
  &= \sum_{n=1}^\infty a_\rep(n) \,\cM^{-1}\mspace{-2mu}\{f\}\mspace{-5mu}\left(\frac{n}{x}\right)
  ,
\end{align}
where $\cM^{-1}$ is the inverse Mellin transform.

\subsubsection*{The approximate functional equation}

Suppose $L(s,\rep)$ satisfies the functional equation
\begin{align*}
  \left(\frac{\qrep}{\pi^\repdim}\right)^{\!\frac{s}{2}}\prod_{j=1}^\repdim \halfGamma{s + \kappa_j} L(s,\rep)
  =
  \rootnum_\rep \left(\frac{\pi^\repdim}{\qrep}\right)^{\!\frac{1-s}{2}}\prod_{j=1}^\repdim \halfGamma{1 - s + \bar\kappa_j} L(1-s,\bar\rep)
  ,
\end{align*}
where $a_{\bar\rep}(n) = \overline{a_\rep(n)}$. Then a simple form of ``the'' approximate functional equation (there are many variations) is
\begin{align}
  \label{eq:afe}
  L(s,\rep) =
  \sum_{n=1}^\infty \frac{a_\rep(n)}{n^s} V_s\!\left(\frac{n}{\sqrt{\qrep}}\right)
  +
  \rootnum_\rep \prod_{j=1}^\repdim \frac{\halfGamma{1 - s + \bar\kappa_j}}{\halfGamma{s + \kappa_j}} \sum_{n=1}^\infty \frac{a_{\bar\rep}(n)}{n^{1-s}} V_{1-s}\!\left(\frac{n}{\sqrt{\qrep}}\right) \left(\frac{\qrep}{\pi^\repdim}\right)^{\!s-\half}
  ,
\end{align}
where $V_s(y) \approx \onelr{y < |s|^{\frac{\repdim}{2}}}$ is some nice function which effectively truncates the Dirichlet series. The above is valid inside the critical strip $0 < \sigma < 1$, for $L(s,\rep)$ holomorphic. (If $L(s,\rep)$ has poles, then some additional terms must be included). Following \cite[Thm.\ 5.4]{IK}, one can take
\begin{align*}
  V_s(y) \coloneqq \frac{1}{2\pi i}\int_{3 - i\infty}^{3 + i\infty} \prod_{j=1}^\repdim \frac{\halfGamma{s+w}}{\halfGamma{s}} \left(\cos\frac{\pi w}{4A}\right)^{-B} y^{-w} \frac{dw}{w}
\end{align*}
with $A,B$ large integers.

\subsubsection*{Combining Mellin inversion and the approximate functional equation}

Let's assume the integrand in \eqref{eq:avg_mellin} is holomorphic in the half plane $\sigma \geqslant c$. Then
\begin{align*}
  \frac{1}{2\pi i} \int_{c_0 - i\infty}^{c_0 + i\infty} \frac{1}{\#\cF} \sum_{\rep \in \cF} L(s,\rep)\,\x^{s-\half} f(s)ds
  &=
  \frac{1}{\#\cF} \sum_{\rep \in \cF} \frac{1}{\sqrt{x}} \sum_{n=1}^\infty a_\rep(n) \cM^{-1}\mspace{-2mu}\{f\}\mspace{-5mu}\left(\frac{n}{x}\right)
\end{align*}
by \eqref{eq:mellin_inversion}, but also
\begin{align*}
  \frac{1}{2\pi i}
  &\int_{c_0 - i\infty}^{c_0 + i\infty} \frac{1}{\#\cF} \sum_{\rep \in \cF} L(s,\rep)\,\x^{s-\half} f(s)ds
  \\
  &=
  \frac{1}{2\pi i} \int_{c - i\infty}^{c + i\infty} \frac{1}{\#\cF} \sum_{\rep \in \cF} L(s,\rep)\,\x^{s-\half} f(s)ds
  \\
  &=
  \frac{1}{2\pi i} \int_{c - i\infty}^{c + i\infty} \frac{1}{\#\cF} \sum_{\rep \in \cF} \sum_{n=1}^\infty \frac{a_\rep(n)}{n^s} V_s\!\left(\frac{n}{\sqrt{\qrep}}\right) \,\x^{s-\half} f(s)ds
  \\
  &+
  \frac{1}{2\pi i} \int_{c - i\infty}^{c + i\infty} \frac{1}{\#\cF} \sum_{\rep \in \cF} \rootnum_\rep \prod_{j=1}^\repdim \frac{\halfGamma{1 - s + \bar\kappa_j}}{\halfGamma{s + \kappa_j}} \sum_{n=1}^\infty \frac{a_{\bar\rep}(n)}{n^{1-s}} V_{1-s}\!\left(\frac{n}{\sqrt{\qrep}}\right) \left(\frac{\pi^\repdim}{\qrep}\right)^{\!s-\half} \x^{s-\half} f(s)ds
\end{align*}
by \eqref{eq:afe}. I.e.,
\begin{align*}
  \frac{1}{\#\cF}
  &\sum_{\rep \in \cF} \frac{1}{\sqrt{x}} \sum_{n=1}^\infty a_\rep(n) \cM^{-1}\mspace{-2mu}\{f\}\mspace{-5mu}\left(\frac{n}{x}\right)
  \\
  &=
  \frac{1}{2\pi i} \int_{c - i\infty}^{c + i\infty} \frac{1}{\#\cF} \sum_{\rep \in \cF} \rootnum_\rep \prod_{j=1}^\repdim \frac{\halfGamma{1 - s + \bar\kappa_j}}{\halfGamma{s + \kappa_j}} \sum_{n=1}^\infty \frac{a_{\bar\rep}(n)}{n^{1-s}} V_{1-s}\!\left(\frac{n}{\sqrt{\qrep}}\right) \left(\frac{\pi^\repdim \x}{\qrep}\right)^{\!s-\half} f(s)ds
  \\
  &+
  \frac{1}{2\pi i} \int_{c - i\infty}^{c + i\infty} \frac{1}{\#\cF} \sum_{\rep \in \cF} \sum_{n=1}^\infty \frac{a_\rep(n)}{n^s} V_s\!\left(\frac{n}{\sqrt{\qrep}}\right) \,\x^{s-\half} f(s)ds
  .
\end{align*}
If we assume that the arithmetic objects $\rep \in \cF$ have $L$-functions with the same root number $\rootnum_\rep \eqqcolon \rootnum_\cF$ and gamma factors, and also all have similar conductor $\qrep \approx q$, then
\begin{align}
  \nonumber
  \frac{1}{\#\cF}
  \sum_{\rep \in \cF}
  &\frac{1}{\sqrt{x}} \sum_{n=1}^\infty a_\rep(n) \,\cM^{-1}\mspace{-2mu}\{f\}\mspace{-5mu}\left(\frac{n}{x}\right)
  \\
  \label{eq:avg_afe_mt}
  &=
  \frac{\rootnum_\cF}{2\pi i} \int_{c - i\infty}^{c + i\infty} \prod_{j=1}^\repdim \frac{\halfGamma{1 - s + \bar\kappa_j}}{\halfGamma{s + \kappa_j}} \frac{1}{\#\cF} \sum_{\rep \in \cF} \sum_{n=1}^\infty \frac{a_{\bar\rep}(n)}{n^{1-s}} V_{1-s}\!\left(\frac{n}{\sqrt{q}}\right) \left(\frac{\pi^\repdim \x}{q}\right)^{\!s-\half} f(s)ds
  \\
  \nonumber
  &+
  \frac{1}{2\pi i} \int_{c - i\infty}^{c + i\infty} \frac{1}{\#\cF} \sum_{\rep \in \cF} \sum_{n=1}^\infty \frac{a_\rep(n)}{n^s} V_s\!\left(\frac{n}{\sqrt{\qrep}}\right) \,\x^{s-\half} f(s)ds
  \\
  \nonumber
  &+
  \frac{\rootnum_\cF}{2\pi i} \int_{c - i\infty}^{c + i\infty} \prod_{j=1}^\repdim \frac{\halfGamma{1 - s + \bar\kappa_j}}{\halfGamma{s + \kappa_j}} \frac{1}{\#\cF} \sum_{\rep \in \cF} \sum_{n=1}^\infty \frac{a_{\bar\rep}(n)}{n^{1-s}}
  \\
  \nonumber
  &\hspace{4cm}
  \cdot\left[V_{1-s}\!\left(\frac{n}{\sqrt{\qrep}}\right)\left(\frac{\pi^\repdim \x}{\qrep}\right)^{\!s-\half} - V_{1-s}\!\left(\frac{n}{\sqrt{q}}\right)\left(\frac{\pi^\repdim \x}{q}\right)^{\!s-\half}\right] f(s)ds
  .
\end{align}

The term \eqref{eq:avg_afe_mt} exhibits scale-invariance that's similar to the examples from the previous section. This is expected: the ratios conjecture uses the approximate functional equation as part of its recipe, and it's the resulting term that has produced the results of the earlier sections; cf.\ \cref{rem:dirichet_functional_equation_factor}.

For example, if $c < \thalf$ and one takes $\cF$ to be $\cF_\chi$, the set of positive fundamental discriminants between $D_0$ and $D$ from \cref{dirichlet_murmuration}, then \eqref{eq:avg_afe_mt} becomes
\begin{align}
  \nonumber
  \frac{\rootnum_\cF}{2\pi i}
  &\int_{c - i\infty}^{c + i\infty} \frac{\halfGamma{1 - s}}{\halfGamma{s}} \frac{1}{\#\cF_\chi} \sum_{d \in \cF_\chi} \sum_{n=1}^\infty \frac{\chi_d(n)}{n^{1-s}} V_{1-s}\!\left(\frac{n}{\sqrt{D}}\right) \left(\frac{\pi \x}{D}\right)^{\!s-\half} f(s)ds
  \\
  \label{eq:afe_t2s}
  &=
  \frac{\rootnum_\cF}{2\pi i} \int_{c - i\infty}^{c + i\infty} \frac{\halfGamma{1 - s}}{\halfGamma{s}} \sum_{n=1}^\infty \frac{1}{n^{2-2s}} \frac{\#\{d \in \cF_\chi \,:\, (n,d) = 1\}}{\#\cF_\chi} V_{1-s}\!\left(\frac{n}{\sqrt{D}}\right) \left(\frac{\pi \x}{D}\right)^{\!s-\half} f(s)ds
  \\
  \nonumber
  &+
  \frac{\rootnum_\cF}{2\pi i}
  \int_{c - i\infty}^{c + i\infty} \frac{\halfGamma{1 - s}}{\halfGamma{s}} \frac{1}{\#\cF_\chi} \sum_{d \in \cF_\chi} \sum_{\substack{n=1 \\ n \neq \square}}^\infty \frac{\chi_d(n)}{n^{1-s}} V_{1-s}\!\left(\frac{n}{\sqrt{D}}\right) \left(\frac{\pi \x}{D}\right)^{\!s-\half} f(s)ds
  .
\end{align}


Estimating \eqref{eq:afe_t2s} can be done using \cite[Lemma 3.9]{quadratictwists}
\begin{align*}
  \#\{0 < d < D \,:\, \text{$d$ a fundamental discriminant},\, (n,d) = 1\} = \frac{3D}{\pi^2} \prod_{p\mid n}\frac{p}{p+1} + O(D^{\half + \eps}n^\eps)
  ,
\end{align*}
and \cite[Lemma 3.12]{quadratictwists}
\begin{align*}
  \sum_{n=1}^\infty \frac{1}{n^{2-2s}} \prod_{p\mid n} \frac{p}{p+1} = \frac{\zeta(2-2s)}{\zeta(3-2s)} \prod_p \left(1 + \frac{1}{(p+1)(p^{s+1}-1)}\right)
  .
\end{align*}

The Dirichlet series above is absolutely convergent, so one might guess that
\begin{align}
  \nonumber
  \frac{\rootnum_\cF}{2\pi i}
  \int_{c - i\infty}^{c + i\infty}
  &\frac{\halfGamma{1 - s}}{\halfGamma{s}} \frac{1}{\#\cF_\chi} \sum_{d \in \cF_\chi} \sum_{n=1}^\infty \frac{\chi_d(n)}{n^{1-s}} V_{1-s}\!\left(\frac{n}{\sqrt{D}}\right) \left(\frac{\pi \x}{D}\right)^{\!s-\half} f(s)ds
  \\
  \label{eq:afe_t2s_2}
  &\approx
  \frac{\rootnum_\cF}{2\pi i} \int_{c - i\infty}^{c + i\infty} \frac{\halfGamma{1 - s}}{\halfGamma{s}} \frac{\zeta(2-2s)}{\zeta(3-2s)} \prod_p \left(1 + \frac{1}{(p+1)(p^{s+1}-1)}\right) \left(\frac{\pi \x}{D}\right)^{\!s-\half} f(s)ds
  \\
  \label{eq:afe_t2ns}
  &+
  \frac{\rootnum_\cF}{2\pi i}
  \int_{c - i\infty}^{c + i\infty} \frac{\halfGamma{1 - s}}{\halfGamma{s}} \frac{1}{\#\cF_\chi} \sum_{d \in \cF_\chi} \sum_{\substack{n < \sqrt{|s|D} \\ n \neq \square}}^\infty \frac{\chi_d(n)}{n^{1-s}} \left(\frac{\pi \x}{D}\right)^{\!s-\half} f(s)ds
  .
\end{align}
The integrand of \eqref{eq:afe_t2s_2} resembles the integrand in \cref{dirichlet_murmuration}.

We now look at \eqref{eq:afe_t2ns}, the sum over non-square $n$. Let's consider the same quantity but with the values $\chi_d(n)$ replaced by iid random variables $X_d(n)$ with mean $0$ and absolute value $1$. The variance of the resulting truncated Dirichlet series is
\begin{align*}
  \frac{(|s|D)^{\sigma - \half - \eps}}{\#\cF}
  \ll
  \mathrm{Var}\!\left(\frac{1}{\#\cF_\chi} \sum_{d \in \cF_\chi} \sum_{\substack{n < \sqrt{|s|D} \\ n \neq \square}}^\infty \frac{X_d(n)}{n^{1-s}}\right)
  \ll
  \frac{(|s|D)^{\sigma - \half + \eps}}{\#\cF}
  .
\end{align*}
Hence, if $f(\sigma + it)$ decays sufficiently quickly as $|t| \to \infty$, then the random variable equivalent of \eqref{eq:afe_t2ns} is asymptotically small with probability $1$ if $\#\cF \gg D^{c - \half + \eps}$, and $c - \thalf$ could not be replaced by any smaller constant. This is a consequence of the central limit theorem \cite[Thm.\ 3.4.1]{durrett}. One can make completely effective the rate of convergence using the Berry--Esseen theorem \cite{tyurin_english}. It's often quite unclear to what extent probabilistic heuristics of this sort resemble the arithmetic, but in this case we recover the reasoning \cite[(2.19) etc.]{conrey_snaith} leading up to the ratios conjecture stated in \cref{ratios_conjecture_dirichlet}.

In our companion paper \cite{quadratictwists}, we go through the analysis we have outlined in this section, and prove unconditionally that, for quadratic twist families of $\mathrm{GL}_1$ automorphic representations, i.e.\ sums
\begin{align*}
  \frac{1}{\#\cF_\chi}
  \sum_{d \in \cF_\chi}
  &\frac{1}{\sqrt{x}} \sum_{n=1}^\infty \,\cM^{-1}\mspace{-2mu}\{f\}\mspace{-5mu}\left(\frac{n}{x}\right) n^{i\tau}\chi(n)\chi_d(n)
\end{align*}
with $\chi$ primitive and $\tau \in \R$, this approach does yield murmurations in the manner described here.

\renewcommand{\bibliofont}{\normalfont\small} 
\bibliographystyle{amsalpha}
\bibliography{murmurationsbib}{}

\end{document}